\documentclass[a4paper,reqno,11pt]{amsart}
\usepackage{amsfonts}
\usepackage{mathrsfs}
\usepackage{cite}
\usepackage{amscd,amsfonts,amssymb, amsmath,stmaryrd,mathrsfs,txfonts}
\usepackage{latexsym}
\usepackage{multirow}
\usepackage{hyperref}
\usepackage{xcolor}
\numberwithin{equation}{section}
\newtheorem{Theorem}{Theorem}[section]
\newtheorem{Lemma}{Lemma}[section]

\theoremstyle{definition}
\newtheorem{Definition}{Definition}[section]
\theoremstyle{remark}
\newtheorem{Remark}{Remark}[section]
\newtheorem{Proposition}{Proposition}[section]

\providecommand{\norm}[1]{\left\Vert#1\right\Vert}
\renewcommand{\r}{\rho}
\renewcommand{\t}{\theta}

\newcommand{\R}{{\mathbb R}}
\newcommand{\Dv}{{\rm div}}
\newcommand{\na}{\nabla}

\newcommand{\Y}{{\mathcal Y}}
\newcounter{RomanNumber}
\newcommand{\MyRoman}[1]{\setcounter{RomanNumber}{#1}\Roman{RomanNumber}}

\def\be{\begin{equation}}
\def\en{\end{equation}}
\def\bs{\begin{split}}
\def\es{\end{split}}

\newcommand{\p}{\partial}

\newcommand{\ess}{{\rm ess\sup}}

\newcommand{\LB}{\left[}
\newcommand{\RB}{\right]}

\allowdisplaybreaks

\title[Martingale solutions for the stochastic Navier-Stokes Equations]
{Martingale solutions for the three-dimensional stochastic
nonhomogeneous incompressible Navier-Stokes equations driven by
L\'{e}vy processes}
\author{Robin Ming Chen, Dehua Wang, Huaqiao Wang}
\address{Department of Mathematics
University of Pittsburgh, Pittsburgh, PA 15260}
\email{mingchen@pitt.edu}
\address{Department of Mathematics
University of Pittsburgh, Pittsburgh, 15260, USA.}
\email{dwang@math.pitt.edu }
\address{Institute of Applied Physics and Computational Mathematics,
                             Beijing 100088, China; and
Department of Mathematics, University of Pittsburgh, Pittsburgh,
15260, USA.} \email{hqwang111@163.com}

\thanks{Corresponding author: hqwang111@163.com}
\bigbreak

\date{September 10, 2017}
\usepackage{hyperref}

\begin{document}
\begin{abstract}
In this paper, the three-dimensional stochastic nonhomogeneous
incompressible Navier-Stokes equations driven by L\'{e}vy process
consisting of the Brownian motion, the compensated Poisson random
measure and the Poisson random measure are considered in a bounded
domain. We obtain the existence of martingale solutions. The
construction of the solution is based on the classical Galerkin
approximation method, stopping time, the compactness method and the
Jakubowski-Skorokhod theorem. 

\end{abstract}

\maketitle

\section{Introduction }
\setcounter{equation}{0}
L\'{e}vy processes were introduced by L\'{e}vy in 1937. They are
often applied to the term structure and credit risk areas. They also
have many important applications including option pricing and the
Black-Scholes formula.
For example, in financial mathematics, the classical model for a stock price is a geometric
Brownian motion. However, wars, decisions
of the federal reserve and other central banks, and other news can
cause the stock price to make a sudden shift. To model this, one
would like to represent the stock price by a L\'evy process which allows for jumps.

Another interesting application of the L\'evy processes can be found in the study of
the stochastic Navier-Stokes equations. The stochastic Navier-Stokes equations have a long history
as a model to understand turbulence in fluid mechanics, structural vibrations in aeronautical applications, and unknown random external forces such as sun heating and industrial pollution in atmospheric dynamics. In real physical situations, the random external forces may exhibit jumps and hence a purely continuous process is not enough to capture the full dynamics. This again motivates the needs for introducing jump processes in the system.

In this work we are concerned with the study of Navier-Stokes equations driven by L\'{e}vy processes.
Let $D\subset \mathbb{R}^3$ be a bounded domain with smooth boundary
$\partial D$ and $\Omega$ be a sample space. We consider the
following system of stochastic PDEs:
\begin{equation}\label{1.1}
\begin{cases}
\rho du+[\rho (u\cdot\nabla) u-\nu\Delta u+\nabla p]dt=\rho
f(t,u)dt+\rho g(t,u)d W+\int_Z\rho Ld\lambda,\\
d\rho+\Dv(\rho u)dt=0,\\
\Dv u=0,
\end{cases}
\end{equation}
in $\Omega\times [0,T]\times D$, with the initial data
\begin{equation}\label{1.2}
\rho|_{t=0}=\rho_0,\; u|_{t=0}=u_0,
\end{equation}
and the homogeneous boundary condition
\begin{equation}\label{1.3}
u|_{\partial D}=0.
\end{equation}
Here $\rho \geq 0$, $u=(u_1,u_2,u_3)\in\mathbb{R}^3$ and $p \in \mathbb{R}$ denote the
density, the velocity and the pressure, respectively; the viscosity
coefficient $\nu$ satisfies $\nu>0$; $\rho f(t,u)$ is the
deterministic external force; $Z$ is a measurable metric space; the random external force is characterized by the L\'{e}vy process $\rho g(t,u)dW + \int_Z\rho L
d\lambda$, where $W$ is an
$\mathbb{R}^d$-valued standard Brownian motion, and
$$Ld\lambda=\left\{\begin{array}{ll} F\left(
u(x,t-),z\right)\tilde{\pi}(dt,dz), & \text{if } |z|_Z<1,\\
G\left( u(x,t-),z\right)\pi(dt,dz), & \text{if } |z|_Z\ge1,\end{array}\right.$$
where
$\pi(dt,dz)$ is a time homogeneous Poisson measure,
$\tilde{\pi}(dt,dz)$ is the compensated Poisson measure associated
to $\pi$ which is defined as  $\tilde{\pi}(dt,dz)=\pi(dt,dz)-dt \mu(dz)$, where
$\mu(\cdot)=\mathbb{E}\pi(1,\cdot)$ is the intensity measure. Here $\mathbb{E}X=\int_{\Omega}XdP$ denotes
the expectation of the stochastic process $X(\omega,t),
\omega\in\Omega$.


\subsection{History of the problem} There have been extensive studies on the nonhomogeneous Navier-Stokes equations. In the deterministic case ($g = L = 0$), Kazhikhov \cite{KAV} obtained weak
solutions for initial density bounded away from zero. Simon \cite{SJ-1990} proved the global existence of strong solutions in two dimensions. For three-dimensional case,
Ladyhzenskaya-Solonnikov \cite{LS}, Padula \cite{PM,PM1} and Salvi
\cite{SR} established the local existence of strong solutions. The uniqueness of strong solutions in $\mathbb{R}^3$ was later proved by Choe-Kim \cite{CK1}.

When $g$ or $L$ does not vanish, the first and second equations in
\eqref{1.1} are stochastic. For viscous compressible flows,
Tornatore \cite{TE} obtained the existence and uniqueness of global solutions for the two-dimensional periodic barotropic fluids with an additive noise, i.e., the random external forcing is
independent of the fluid velocity $u$.
Feireisl-Maslowski-Novotn\'y in \cite{FMN} later
considered the three-dimensional problem in Sobolev spaces where the noise, under suitable weak formulation, can be regarded as an additive one. They managed to show the existence of strong solutions by using an abstract
measurability theorem \cite{BT-1973} and proved that the weak solution generates a random variable. When the noise is {\it multiplicative}, that is, the random external forcing depends on $u$, the problem becomes more involved. Some recent development on the existence of martingale solutions can be found in \cite{BH-2014,SSA,WW}.

For incompressible nonhomogeneous fluids, the existence of
martingale solutions to the equations (1.1)-(1.3) driven by an additive
noise was established
by Yashima \cite{YHF} with positive initial density. For general
multiplicative noise, Cutland-Enright \cite{CE} constructed strong
solutions by using the Loeb space techniques in two and
three-dimensional bounded domains. For the wellposedness of the homogeneous
stochastic incompressible Navier-Stokes equations, see
\cite{BA,BT-1973,BH-2000,BM,CG-1994,CP-2001,CP-1997,DZ,FG,FM-1995,
HZ,JKU,ME,ME-2014,MS-2002,MB-2004,MB-2005,SM,SS,TT} and the
references therein. See
\cite{BH-1999,CDK-2012,CNV-2014,CI-2011,PZ-1996, DA-2013,EMS-2001,
EH-2001,FGM-2008,HSV-2013,
HV-2014,HM-2011,HM-2006,HMW-2014,HP-2014,KMS-2000,LS-2000,MB-2004}
and the references therein for the studies and results on the
incompressible stochastic Euler equations, ergodicity of stochastic
partial differential equations, stochastic equations for turbulent
flows, stochastic conservation laws, and so on.

\subsection{Main results} In this paper we consider the existence of martingale solutions to
the three-dimensional stochastic nonhomogeneous incompressible
Navier-Stokes equations with L\'{e}vy processes. Our approach is
based on the Galerkin approximation scheme and the compactness
method. We will outline the main idea in the later part of the
section.

First, we define the concept of solutions for the problems
\eqref{1.1}-\eqref{1.3} as follows.
\begin{Definition}\label{definition1.1}
A martingale solution of \eqref{1.1}-\eqref{1.3} is a system
$((\Omega,\mathscr{F},\mathscr{F}_t,P), W, \pi,\rho, u)$, which
satisfies

(1) $(\Omega,\mathscr{F},\mathscr{F}_t,P)$ is a filtered probability
space with a filtration $\mathscr{F}_t$, i.e., a set of sub
$\sigma$-fields of $\mathscr{F}$ with
$\mathscr{F}_s\subset\mathscr{F}_t\subset\mathscr{F}$ for $0\le
s<t<\infty$,

(2) $W$ is a $d$-dimensional $\mathscr{F}_t$ standard Brownian
motion,

(3) $\pi$ is a time homogeneous Poisson random measure over
$(\Omega,\mathscr{F},\mathscr{F}_t,P)$ with the intensity measure
$\mu$,

(4) for almost every $t$, $\rho(t)$ and $u(t)$ are 
progressively measurable,

(5) $\rho\in L^\infty(\Omega, L^\infty(0,T; L^\infty(D)))$, $u\in
L^4(\Omega, L^\infty(0,T; H))\cap L^2(\Omega,L^2(0,T; V))$. For all
$t\in[0,T]$,  any $\varphi\in H^1(D)$ and $\psi\in V$ (see
\eqref{2.1b} for definition of $V$), the following holds $P$-a.s.
\begin{equation}\label{1.8}
\int_{D}\rho(t)\varphi dx-\int_{D} \rho_0\varphi
dx=\int_0^t\int_{D}\rho u\cdot\nabla\varphi dxds,
\end{equation}
and
\begin{equation}\label{1.9}
\begin{split}
\int_{D}(\rho u)(t)&\psi dx-\int_0^t\int_{D}\left(\rho u u
\nabla \psi-\nu\nabla u\cdot\nabla\psi\right)dxds-\int_{D}\rho_0u_0\psi dx\\
&=\int_0^t\int_{D}\rho f(s,u)\psi
dxds+\int_0^t\int_{D}\rho g(s,u)\psi dxdW\\
&\quad+\int_0^t\int_{|z|_{Z}<1}\int_D\rho
F\!\left( u(x,s-),z\right)\psi dx\tilde{\pi}(ds,dz)\\
&\quad+\int_0^t\int_{|z|_{Z}\ge1}\int_D\rho G\!\left(
u(x,s-),z\right)\psi dx\pi(ds,dz),
\end{split}
\end{equation}
and
\begin{equation}\label{1.10}
\rho|_{t=0}=\rho_0, \int_{D}\rho(0)u(0)\varphi
dx=\int_{D}\rho_0u_0\varphi dx.
\end{equation}
In the above, all stochastic integrals are defined in the sense of It\^{o}, see
\cite{AD,PZ-1992,ELC,KHH,KS-1991,PZ,RST}.
\end{Definition}

\medskip

Throughout this paper, we assume that the Brownian motion $W$ is independent of the compensated
Poisson measure $\tilde{\pi}(dt,dz)$. The intensity
measure $\mu$ on $Z$ satisfies the conditions $\mu(\{0\})= 0$,
$\int_Z(1\wedge|z|^2)\mu(dz)<\infty$ and $\int_{|z|_Z\ge1} |z|^p
\mu(dz)<\infty,\forall p\ge 1$.  We also assume that $\{\mathscr{F}_t\}$ is a right continuous filtration
over the probability space $(\Omega,\mathscr{F},P)$ such that
$\mathscr{F}_0$ contains all $P$-negligible subsets of $\Omega$.

Before we state our main theorem, we make the following assumptions on the external forces.

{\noindent}{\bf Assumption (A)}. Assume that $f: (0,T)\times
H\rightarrow H$ and $g: (0,T)\times H\rightarrow H^{\times d}$ are
continuous and nonlinear mappings, which satisfy the following
condition: there exists a positive constant $C$ such that
$$\norm{f(t,u)-f(t,v)}_{L^2(D)}\le C \norm{u-v}_{L^2(D)},\; \norm{g(t,u)-g(t,v)}_{L^2(D)}\le C \norm{u-v}_{L^2(D)},$$
\begin{equation*}
\|f(t,u)\|_{L^2(D)}\leq C\left(1+\norm{u}_{L^2(D)}\right),\;
\|g(t,u)\|_{L^2(D)}\leq C\left(1+\norm{u}_{L^2(D)}\right),
\end{equation*}
where $H^{\times d}$ is the product of $d$ copies of the space $H$ which is defined in \eqref{2.1a}.

{\noindent}{\bf Assumption (B)}. For all $t\in[0,T]$, there exists a
positive constant $C$ such that
\begin{equation}\label{1.6}
\begin{split}
\int_{|z|_Z<1}&\norm{F(u,z)-F(
v,z)}_{L^2(D)}^2\mu(dz)+\int_{|z|_Z\ge1}\norm{G(
u,z)-G(v,z)}_{L^2(D)}^2\mu(dz)\\
&\!\!\!\leq C\norm{u- v}_{L^2(D)}^2. 
\end{split}
\end{equation}
For each $p\ge 2$ and all $t\in[0,T]$, there exists a positive
constant $C$ such that
\begin{equation}\label{1.7}
\begin{split}
\int_{|z|_Z<1} \norm{F( u,z)}_{L^2(D)}^p\mu(dz)+ \int_{|z|_Z\ge1}
\norm{G(u,z)}_{L^2(D)}^p\mu(dz)\leq C
\left(1+\|u\|^p_{L^2(D)}\right). 
\end{split}
\end{equation}

Our main results are the following.
\begin{Theorem}\label{Theorem1.1}
Let the assumptions (A) and (B) be satisfied and assume that $u_0\in
L^2(D),\ \r_0\in L^\infty(D)$ satisfying $0<m\le\r_0\le M$. Then there exists a
martingale solution of problems \eqref{1.1}-\eqref{1.3} in the sense
of Definition \ref{definition1.1}.
\end{Theorem}

\subsection{Outline of ideas} Theorem \ref{Theorem1.1} will be proved through the following steps. First we use the Faedo-Galerkin method to construct the
approximate solutions to the problem \eqref{1.1}-\eqref{1.3}.
More precisely, on the probability space
$(\Omega,\mathscr{F},P)$ with a given $d$-dimensional Brownian
motion $W$ and Poisson random measure $\pi$, for the finite-dimensional approximate
system we use the
Picard iteration to obtain a local solution
$(W_n,\pi_n,\rho^n,u^n)$ in a short time interval $[0, T_n]$.
Here, different from the deterministic situation, the velocity $u$ in general exhibits jump discontinuity (in time), and hence one cannot apply the standard method of characteristics to solve the transport equation for $\rho$. To overcome this difficulty, we adapt the result of DiPerna-Lions \cite{DP} on transport theory for less regular vector fields $u$ to obtain a solution $\rho \in L^\infty$. To obtain a uniform time interval $[0, T]$ of existence for all $n$, we need to derive the energy estimates. This can be done by applying
the stopping time and the Burkholder-Davis-Gundy
inequality.

The second step is to take a limit as $n\to \infty$ and prove the existence of Martingale solutions. From energy estimates, the approximate solutions $(W_n,\pi_n,\rho^n,u^n)$ may converge on $[0, T]$. However the convergence is too weak to guarantee that the limit is a solution on $[0, T]$. In the two-dimensional case, it can be shown by using certain monotonicity principle that the nonlinear terms converge to the right limit and hence a global strong solution can be obtained \cite{MS-2002}. But when the space dimension is three the monotonicity does not hold and to the best of our knowledge there is no result on the global strong solutions. This is why we pursue instead the Martingale solutions. As is explained, the main issue is the convergence of the nonlinear terms.

To this end, we relax the restriction on the probability space and
aim to prove a tightness result of the random variables
$(W_n,\pi_n,\rho^n,u^n)$. This can be obtained by applying the
Arzsela-Ascoli's Theorem combined with the Aubin-Simon Lemma
\cite{SJ}. Moreover in order to analyze the nonlinear terms, we
prove the tightness of $\rho^nu^n$ as well. Then from the
Jakubowski-Skorokhod Theorem \cite{JA} there exist a probability
space $(\mathring{\Omega}, \mathring{\mathscr{F}},\mathring{P})$ and
random variables
$(\mathring{W}_{n_j},\mathring{\pi}_{n_j},\mathring{\rho}^{n_j},\mathring{u}^{n_j},
\mathring{\rho}^{n_j}\mathring{u}^{n_j})\to (W,\pi,\rho,u,h)\ \ \!\!
\mathring{P}-a.s.$, with the property that the probability
distribution of $(\!\mathring{W}_{n_j}, \mathring{\pi}_{n_j},
\mathring{\rho}^{n_j},\mathring{u}^{n_j},\mathring{\rho}^{n_j}\mathring{u}^{n_j})$
is the  same as  that of $(W_n,\pi_n,\rho^n,u^n,\r^nu^n)$. By using
a cut-off function we can also show that the random variables
$(\mathring{W}_{n_j},\mathring{\pi}_{n_j},\mathring{\rho}^{n_j},\mathring{u}^{n_j})$
satisfy the approximate equations in $(\mathring{\Omega},
\mathring{\mathscr{F}},\mathring{P})$. When passing to a limit as
$n\to\infty$, the usual method is to show that the limit process of
the stochastic integral is a martingale, and to identify its
quadratic variation. Then apply the representation theorem for
martingales or the revised representation theorem (see \cite{JA}) to
prove that it solves the equations. But here instead, we can prove
that $W$ is a Brownian motion and $\pi$ is a time homogeneous
Poisson random measure. Then in view of the uniform integrability
criterion, Vitali's convergence theorem and mollification
techniques, together with the almost sure convergence on
$(\mathring{\Omega}, \mathring{\mathscr{F}},\mathring{P})$, we can
obtain that $(W,\pi,\r,u)$ satisfies the equations
\eqref{1.1}-\eqref{1.3} by passing to the limit directly. Therefore
it is a martingale solution of \eqref{1.1}-\eqref{1.3} in the sense
of Definition \ref{definition1.1}.

The rest of the paper is organized as follows. We recall some
analytic tools in Sobolev spaces and some basic theory of stochastic
analysis in Section 2.  In Section 3, we construct the
solutions to an approximate scheme by the Faedo-Galerkin method. In
Section 4, we prove the tightness property of the approximate solutions $(W_n,\pi_n,\rho^n,u^n)$ and then pass to
the limit as $n\to\infty$.

\medskip

\noindent{\bf Notation.} Throughout the paper we drop the parameter
$\omega\in\Omega$. Moreover, we use $C$ to denote a generic constant
which may vary in different estimates. \if false For simplicity, we
will write $A\lesssim B$ if $A\leq C B$. \fi

\section{Preliminaries}
\setcounter{equation}{0}
Let $H^1(D)$ denote the Sobolev space of all $u\in L^2(D)$ for
which there exist weak derivatives ${\partial u\over \partial
x_i}\in L^2(D), i=1,2,\ldots,d$.  Let $C_c^\infty(D)$ denote the
space of all $\mathbb{R}^d$ valued functions of class $C^\infty$
with compact supports contained in $D$ and define
\begin{equation*}
\mathscr{V} :=\{u\in C_c^\infty(D):\Dv u=0\},
\end{equation*}
\begin{align}\label{2.1a}
 H :=\mbox{the closure of}\ \ \mathscr{V} \ \ \mbox{in}\ \ L^2(D),
\end{align}
\begin{align}\label{2.1b}
 V :=\mbox{the closure of}\ \
\mathscr{V} \ \ \mbox{in}\ \ H^1(D).
\end{align}

In the space $H$, we consider the scalar product and the norm
inherited from $L^2(D)$ and denote them by
$\langle\cdot,\cdot\rangle_H$ and $|\cdot|_H$ respectively, i.e.
\begin{equation*}
\langle u,v\rangle_H=\langle u,v\rangle_{L^2(D)},\; \quad
|u|_H=\norm{u}_{L^2(D)},\; \quad u,v\in H.
\end{equation*}
In the space $V$ we consider the scalar product inherited from
$H^1(D)$, that is
\begin{equation*}
\langle u,v\rangle_V=\langle u,v\rangle_H+\langle\nabla u,\nabla
v\rangle_{L^2(D)}.
\end{equation*}
Let $p_*$ denote the Sobolev conjugate in $\mathbb{R}^3$ which is defined as
\begin{equation*}
p_* := \left\{\begin{array}{ll}
\displaystyle {3p \over 3 - p}, \quad & \text{if } 1\leq p < 3, \\\\
\text{any finite non-negative real number}, & \text{if } p = 3, \\
\infty, & \text{if } p>3.
\end{array}\right.
\end{equation*}
We first recall some properties of products in Sobolev spaces $W^{1,p}(D)$ with $p\ge1$.
%
\begin{Lemma}\label{Lemma2.1}
For $1\le p\le q\le \infty$, $f\in W^{1,p}(D) $ and $g\in
W^{1,q}(D)$, if $r\ge1$ and ${1\over r}={1\over p}+{1\over q_*}$,
then $fg \in W^{1,r}(D)$ and
$$
\norm{fg}_{W^{1,r}(D)}\le \norm{f}_{ W^{1,p}(D)}\norm{g}_{
W^{1,q}(D)}.
$$
For $h\in W^{-1,q}(D)$, if ${1\over p}+{1\over q}\le 1$
and ${1\over r}={1\over p_*}+{1\over q}$, then $fh \in W^{-1,r}(D)$ and
\[
\norm{fh}_{W^{-1,r}(D)}\le \norm{f}_{ W^{1,p}(D)}\norm{h}_{
W^{-1,q}(D)}.
\]
\end{Lemma}
\begin{Lemma}[\cite{LLJ}]\label{Lemma2.2}
Let $(g_k)_{k=1,2,\dots}$ and $g$ be functions in $L^q(0,T; L^q(D))$ for
$q\in(1,\infty)$ such that $\norm{g_k}_{L^q(0,T; L^q(D))}\le C$ for
any $k$ and $g_k\to g$ almost everywhere in $Q_T:=D\times[0,T]$ as
$k\to\infty$. Then $g_k$ converges weakly  to $g$ in $L^q(0,T;
L^q(D))$.
\end{Lemma}

For a probability space $(\Omega,\mathscr{F},P)$ and a Banach space
$X$, denote by  $L^p(\Omega,L^q(0,T;X))\, (1\le p,q<\infty)$ the
space  of random functions defined on $\Omega$ with value in
$L^q(0,T; X)$,  endowed with the norm:
\begin{equation*}
\|u\|_{L^p(\Omega,L^q(0,T;X))}=\left(\mathbb{E}\ \|u\|^p_{L^q(0,T;X)}\right)^{1\over
p}.
\end{equation*}
If $q=\infty$, we write
\begin{equation*}
\|u\|_{L^p(\Omega,L^\infty(0,T;L^q(X)))}=\left(\mathbb{E}\ \ess_{0\le t\le
T}\|u\|^p_{L^q(X)}\right)^{1\over p}.
\end{equation*}
\begin{Remark}\label{Remark2.1}
Note that the result of Lemma \ref{Lemma2.2} also holds for the space
$L^q(\Omega; L^q(0,T; D))$ in $\Omega\times Q_T$.
\end{Remark}

We now list a few preliminary results of stochastic analysis and
useful tools for the sake of convenience and completeness. For
details, we refer the readers to
\cite{AD,PZ-1992,ELC,KHH,KS-1991,PZ,RST} and the references therein.
In particular, we will introduce the definitions of \if false
Brownian motion, progressively measurable process, \fi time
homogenous Poisson random measure, L\'{e}vy process, \if false
martingale, \fi stopping time, It\^{o} formula and the BDG
inequality and so on.

\if false
\begin{Definition}\label{Def2.1}
A stochastic process $W(t,\omega)$ is called a {\em Brownian motion}
if it satisfies the following conditions:

(1) $P\{\omega: W(0,\omega)=0\}=1.$

(2) For any $0\le s\le t$, the random variable $W(t)-W(s)$ is
normally distributed with mean 0 and variance $t-s$, i.e., for any
$a<b$,
\begin{equation*}
P\left\{a\le W(t)-W(s)\le b\right\}={1\over\sqrt{2\pi(t-s)}}\int_a^b
e^{-x^2\over2(t-s)}dx.
\end{equation*}

(3) $W(t,\omega)$ has independent increments, i.e., for any $0\le
t_1<t_2<\ldots<t_n$, the random variables $W(t_1),
W(t_2)-W(t_1),\ldots, W(t_n)-W(t_{n-1})$ are independent.

(4) Almost all simple paths of $W(t,\omega)$ are continuous
functions, i.e.,
\begin{equation*}
P\left\{\omega: W(\cdot,\omega) \mbox{ is continuous}\right\}=1.
\end{equation*}
\end{Definition}

\begin{Definition}\label{Def2.2}  
A stochastic process $X$ is called {\em progressively measurable}
with respect to the filtration $\{\mathscr{F}_t\}$ if for each
$t\ge0$ and $A\in\mathscr{B}(\mathbb{R}^n)$ (Borel sets on
$\mathbb{R}^n$), the set $\{(s,\omega):0\le s\le t, \omega\in\Omega,
X_s(\omega)\in A\}$ belongs to the product $\sigma$-field
$\mathscr{B}([0,t])\otimes\mathscr{F}_t$; in other words, if the
mapping $(s,\omega)\mapsto
X_s(\omega):([0,t]\times\Omega,\mathscr{B}([0,t])\otimes\mathscr{F}_t)\rightarrow
(\mathbb{R}^n,\mathscr{B}(\mathbb{R}^n))$ is measurable for each
$t\ge0$.
\end{Definition}
\fi

\begin{Definition}\label{def_filtration}
A {\em filtration} on the parameter set $\mathbb{T}$ is an increasing
family $\{\mathscr{F}_t: t\in \mathbb{T}\}$ of $\sigma$-algebra. A
stochastic process $X_t,\ t\in \mathbb{T}$ is said to be adapted to
$\{\mathscr{F}_t: t\in \mathbb{T}\}$ if for each $t$, the random
variable $X_t$ is $\mathscr{F}_t$-measurable.
\end{Definition}

Denote $\bar{\mathbb{N}}:=\mathbb{N}\cup\{\infty\}$,
$\mathbb{R}_+:=[0,\infty)$. Let $(Z,\mathscr{Z})$ be a measurable
space. Then by $M(Z)$ we denote the set of all real valued measures
on $(Z,\mathscr{Z})$, and $\mathscr{M}(Z)$ denotes the
$\sigma$-field on $M(Z)$ generated by functions $i_B:
\mu\mapsto\mu(B)\in\mathbb{R}$ for $\mu\in M(Z),B\in\mathscr{Z}$.
Next, we denote the set of all non-negative measures on $Z$ by
$M_+(Z)$, and $\mathscr{M}_+(Z)$ denotes the $\sigma$-field on
$M_+(Z)$ generated by functions $i_B:
M_+(Z)\ni\mu\mapsto\mu(B)\in\mathbb{R}_+$, $B\in\mathscr{Z}$. Finally,
by $M_{\bar{\mathbb{N}}}(Z)$ we denote the family of all
$\bar{\mathbb{N}}$-valued measures on $(Z,\mathscr{Z})$, and
$\mathscr{M}_{\bar{\mathbb{N}}}(Z)$ denotes the $\sigma$-field on
$M_{\bar{\mathbb{N}}}(Z)$ generated by functions $i_B:
M_{\bar{\mathbb{N}}}\ni\mu\mapsto\mu(B)\in\bar{\mathbb{N}}, B\in\mathscr{Z}$.
\begin{Definition}\label{defA.2}
Let $(Z,\mathscr{Z})$ be a measurable space and
$\mu\in\mathscr{M}_+(Z)$. A measurable function $\pi:
(\Omega,\mathscr{F})\rightarrow(M_{\bar{\mathbb{N}}}(Z\times\mathbb{R}_+),\mathscr{M}_{\bar{\mathbb{N}}}(Z\times\mathbb{R}_+))
$ is called a {\em time homogenous Poisson random measure} on
$(Z,\mathscr{Z})$ over $(\Omega,\mathscr{F},\mathscr{F}_t,P)$ if and
only if the following conditions are satisfied

(1) for each $B\in\mathscr{Z}\otimes\mathscr{B}(\mathbb{R}_+),\;
\pi(B):=i_B\circ\pi: \Omega\rightarrow\bar{\mathbb{N}}$ is a Poisson
random variable with parameter $\mathbb E\pi(B)$ (If $\mathbb E \pi(B)=\infty$, then
$\pi(B)=\infty$ );

(2) $\pi$ is independently scattered, that is, if the sets
$B_j\in\mathscr{Z}\otimes\mathscr{B}(\mathbb{R}_+), j=1,2,\dots,n$
are pair-wise disjoint, then the random variables $\pi(B_j),
j=1,2,\dots,n $ are pair-wise independent;

(3) for all $B\in\mathscr{Z}$ and $I\in\mathscr{B}(\mathbb{R}_+)$,
$\mathbb E[\pi(B\times I)]=\lambda(I)\mu(B)$, where $\lambda$ is Lebesgue
measure;

(4) for each $U\in\mathscr{Z}$, the $\bar{\mathbb{N}}$-valued
process $(N(t,U))_{t\ge 0}$ defined by $N(t,U):=\pi(U\times(0,t]),
t\ge 0$ is $\mathscr{F}_t$-adapted and its increments are
independent of the past, i.e. if $t>s\ge0$, then
$N(t,U)-N(s,U)=\pi(U\times(s,t])$ is independent of $\mathscr{F}_s$.
\end{Definition}

Now  we turn to the definition of a L\'{e}vy process.
\begin{Definition}
Let $\mathbb{B}$ be a Banach space. A stochastic process $L=\{L(t):
t\ge 0\}$ over $(\Omega,\mathscr{F},\mathscr{F}_t,P)$ is called an
$\mathbb{B}$-valued {\em L\'{e}vy process}  if the following
conditions are satisfied.

(1) $L(t)$ is $\mathscr{F}_t$-measurable for any $t\ge 0$;

(2) the random variable $L(t)-L(s)$ is independent of
$\mathscr{F}_s$ for any $0\le s<t$;

(3) $L(0)=0$ a.s.;

(4) For all $0\le s<t$, the law of $L(t+s)-L(s)$ does not depend on
$s$;

(5) $L$ is stochastically continuous;

(6) the trajectories of $L$ are c\`{a}dl\`{a}g in $\mathbb{B}$
$P$-a.s., i.e. which are right-continuous with left limits.
\end{Definition}

Note that we can construct a corresponding Poisson random measure
from a L\'{e}vy process. For example, given a $\mathbb{B}$-valued
L\'{e}vy process over $(\Omega,\mathscr{F},\mathscr{F}_t,P)$, one
can construct an integer-valued random measure in the following way: for
each $(B, I)\in
\mathscr{B}(\mathbb{R})\times\mathscr{B}(\mathbb{R}_+)$, define
\begin{equation*}
\pi_L(B, I):=\sharp\{t\in I~|~\Delta_t L\in B\}\in\bar{\mathbb{N}}.
\end{equation*}
where $\Delta_t L(t):=L(t)-L(t-)=L(t)-\lim_{s\uparrow t}L(s), t>0$
and $\Delta_0 L:=0$. If $\mathbb{B}=\mathbb{R}^d$, then $\pi_L$ is a
time homogeneous Poisson random measure, for details see \cite[Chapter 4, Theorem
19.2]{SKI}.  Conversely, given a Poisson random
measures, we can also construct a corresponding L\'{e}vy process.

\if false
\begin{Definition}\label{Def2.3}
Let $X_t(\omega)=X(\omega, t)$ be a scalar stochastic process
adapted to a filtration $\mathscr{F}_t$ and $E|X_t|<\infty$ for all
$t\in \mathbb{T}$, and $E(X_t|\mathscr{F}_s)$ be the conditional
expectation of $X_t$ given $\mathscr{F}_s$. Then $X_t$ is called a
{\em martingale} with respected to $\mathscr{F}_t$ if for any $s, t
\in \mathbb{T}$, $s\le t$, $E(X_t|\mathscr{F}_s)=X_s$ a.s.. If
$E(X_t|\mathscr{F}_s)\ge (\le) X_s$ a.s.,  $X_t$ is called a
submartingale (supmartingale).
\end{Definition}
\fi
\begin{Definition}\label{Def2.4}
A random variable $\tau(\omega)$ with values in the parameter set
$\mathbb{T}$ is a {\em stopping time} of the filtration
$\mathscr{F}_t$ if $\{\omega: \tau(\omega)\le t\}\in \mathscr{F}_t$
for each $t\in \mathbb{T}$.
\end{Definition}

\if false
\begin{Definition}\label{Def2.5}
An $\mathscr{F}_t$ adapted stochastic process $X_t, t\in [a,b]$ is
called a {\em local martingale} with respect to $\mathscr{F}_t$ if
there exists a sequence of stopping times $\tau_n$, $n=1, 2,
\ldots$,  such that

(1) $\tau_n$ increases monotonically to $b$ almost surely as
$n\to\infty$;

(2) For each $n$, $X_{t\wedge{\tau_n}}$ is a martingale with respect
to $\left\{\mathscr{F}_t, t\in [a,b]\right\}$.
\end{Definition}

\begin{Lemma}\label{Thm2.1}
Let $B(t)$ be a Brownian motion,  $f(t)$  be a random function
adapted to the filtration $\mathscr{F}_t$,  and
$\int_a^b|f(t)|^2dt<\infty$ almost surely. Then the stochastic
process $X_t=\int_a^t f(s)dB(s)$, $t\in[a,b]$ is a local martingale
with respect to the filtration $\left\{\mathscr{F}_t, t\in
[a,b]\right\}$.
\end{Lemma}
\fi
 Let us now recall the It\^{o} formula for general L\'{e}vy-type
stochastic integrals, see \cite{AD,PZ,RST}. 
We define $\mathcal{P}_2(T, \mathbb{B})$ to be the set of all
equivalence classes of mappings $f: [0, T]\times \mathbb{B}\times
\Omega\to\mathbb{R}$ which coincide almost everywhere with respect
to $\varrho\times P$ and which satisfy the following conditions:

(1) $f$ is predictable;

(2) $P\left(\int_0^T\int_{\mathbb{B}}  |f(t,x)|^2\varrho(dt,dx)<\infty\right)=1$.\\
Here $\varrho(t,A)= m((0,t]\times A)$, where $m$ is standard
Lebesgue measure. With this we are ready to give the It\^{o} formula
\cite[Page 251, Theorem 4.4.7]{AD} for general L\'{e}vy-type
stochastic integrals. Let $X$ be the following process
\begin{equation}\label{Represent1}
dX(t)=G(t)dt+F(t)dB(t)+\int_{|x|<1}
H(t,x)\tilde{\pi}(dt,dx)+\int_{|x|\ge1} K(t,x)\pi(dt,dx),
\end{equation}
where for each $t\ge 0$, $|G|^{1\over2},F\in \mathcal{P}_2(T)$ and
$H\in\mathcal{P}_2(T, \mathbb{B})$. Furthermore, we take $\mathbb{B}=\{x\in \mathbb{R}^d:
0<|x|<1\}$ and $K$ to be
predictable so that the left-continuous $\mathscr{F}_t$-adapted
processes $K_t(\omega):[0,\infty)\times \Omega\to
K_t(\omega)\in\mathbb{R}^d$ is measurable, where $\omega\in\Omega$,
see \cite[Page 7]{RST}. Denote 
\begin{equation*}
dX_c(t)=G(t)dt+F(t)dB(t),
\end{equation*}
and
\begin{equation*}
dX_d(t)=\int_{|x|<1} H(t,x)\tilde{\pi}(dt,dx)+\int_{|x|\ge1}
K(t,x)\pi(dt,dx),
\end{equation*}
so that for each $t\ge 0$, we have
\begin{equation*}
X(t)=X(0)+X_c(t)+X_d(t).
\end{equation*}
 Assume that for all
$t>0$, $\sup_{0\le s\le t, \ 0<|x|<1}H(s,x)<\infty$ a.s., then
one has
\begin{Lemma}[It\^{o}'s formula,\cite{AD}]\label{Ito}
If $X$ is a L\'{e}vy-type stochastic integral of the form
\eqref{Represent1}, then for each $\Phi\in C^2(\mathbb{R}^n),\;
t\ge0$, with probability 1 we have
\begin{equation}
\begin{split}
\Phi&(X(t))-\Phi(X(0))=\int_0^t
\partial_i \Phi(X(s-))dX_c^i(s)+{1\over2}\int_0^t
\partial_i\partial_j\Phi(X(s-))d[X_c^i,X_c^j](s)\\
&+\int_0^t \int_{|x|\ge
1}\left[\Phi(X(s-)+K(s,x))-\Phi(X(s-))\right]\pi(ds,dx)\\
&+\int_0^t \int_{|x|<1
}\LB \Phi(X(s-)+H(s,x))-\Phi(X(s-))\RB \tilde{\pi}(ds,dx)\\
&+\int_0^t\int_{|x|<1
}\left[\Phi(X(s-)+H(s,x))-\Phi(X(s-))-H^i(s,x)\partial_i
\Phi(X(s-))\right]\mu(dx)ds
\end{split}
\end{equation}
\end{Lemma}

We now recall the following so-called BDG inequality in stochastic
analysis, see\cite[Page 37, Theorem 3.50]{PZ}.
\begin{Lemma}[Burkholder-Davis-Gundy inequality]\label{Thm2.3}  Let $T>0$, for every fixed $p\ge1$, there is a
constant $C_p\in(0,\infty)$ such that for every real-valued square
integrable c\`{a}adl\`{a}g martingale $M_t$ with $M_0=0$, and for
every $T\ge0$,
\begin{equation*}
C_p^{-1}\mathbb E\left(\langle M\rangle^{p\over 2}_T\right)\le
\mathbb E\left(\max_{0\le t\le T}|M_t|^p\right)\le C_p \mathbb E\left(\langle
M\rangle^{p\over 2}_T\right),
\end{equation*}
where $\langle M\rangle_t$ is the quadratic variation of $M_t$ and
the constant $C_p$ does not depend on the choice of $M_t$.
\end{Lemma}

\begin{Definition}\label{Definition2.9}
Let $\mathbb{B}$ be a separable Banach space and let
$\mathscr{B}(\mathbb{B})$ be its Borel sets. A family of probability
measures $\mathbb{P}$ on $(\mathbb{B},\mathscr{B}(\mathbb{B}))$ is
{\em tight} if for any $\varepsilon>0$, there exists a compact set
$K_\varepsilon\subset \mathbb{B}$ such that $\Pi(K_\varepsilon)\ge
1-\varepsilon$ for all $\Pi\in\mathbb{P}$. A sequence of measures
$\{\Pi_n\}$ on $(\mathbb{B},\mathscr{B}(\mathbb{B}))$ is weakly
convergent to a measure $\Pi$ if for all continuous and bounded
functions $h$ on $\mathbb{B}$
\begin{equation*}
\lim\limits_{n\to\infty}\int_\mathbb{B}
h(x)\Pi_n(dx)=\int_\mathbb{B} h(x)\Pi(dx).
\end{equation*}
\end{Definition}

\begin{Lemma}[Jakubowski-Skorokhod Theorem \cite{JA}]\label{lemma1.4}
 Let $\mathscr{X}$ be a
topological space such that there exists a sequence $\{h_m\}$ of
continuous functions $h_m: \mathscr{X}\to\mathbb{R}$ that separate
points of \,$\mathscr{X}$. Denote by $\mathscr{S}$ the
$\sigma$-algebra generated by the maps $\{h_m\}$. Then

(1) every compact subset of \, $\mathscr{X}$ is metrizable.

(2) every Borel subset of a $\sigma$-compact set in $\mathscr{X}$
belongs to $\mathscr{S}$.

(3) every probability measure supported by a $\sigma$-compact set in
$\mathscr{X}$ has a unique Radon extension to the Borel
$\sigma$-algebra on $\mathscr{X}$.

(4) if \,$\Pi_m$ is a tight sequence of probability measures on
$(\mathscr{X},\mathscr{S})$,   there exist  a subsequence
$\Pi_{m_k}$ converging weakly to a probability measure $\Pi$, and a
probability space $(\Omega,\mathscr{F},P)$  with $\mathscr{X}$
valued Borel measurable random variables $X_{k}$ and $X$ such that,
$\Pi_{m_k}$ is the distribution of $X_k$,
and $X_k \to X$ a.s. on  $\Omega$. Moreover, the law of $X$ is a
Radon measure.
\end{Lemma}
\setcounter{equation}{0}
\section{The Galerkin approximation and a priori estimates}
\subsection{The Galerkin approximation}
\hspace{2mm}

On the probability space $(\Omega,\mathscr{F},P)$ with a given
$d$-dimensional Brownian motion $W$ and Poisson random measure
$\pi$. In order to solve \eqref{1.1}-\eqref{1.3}, we first consider
a suitable orthogonal system formed by a family of smooth functions
$w_n$ vanishing on $\partial D$. One can take the eigenfunctions of
the Dirichlet problem for the Laplacian operator:
\begin{equation*}
-\Delta w_n=\lambda_nw_n \ \ \mbox{on} \ \ D,\quad w_n|_{\partial D}=0.
\end{equation*}
Now, we consider a sequence of finite dimensional spaces
\begin{equation*}
X_n=\mbox{span}\{w_j\}^n_{j=1},\; n=1,2,\ldots.
\end{equation*}
We shall look for the sequences $(\r^n,u^n)$ satisfying the integral
equation:
\begin{align}\label{3.1}
\notag&\int_D \r^n du^n(t)\psi dx-\int_D \r_0 u_0\psi
dx-\int_0^t\int_D\left[\nu\Delta u^n-\r^n u^n\na
u^n)\right]\psi dxds\\
\notag&=\int_0^t\int_D\r^n f(s,u^n)\psi dxds+\int_0^t\int_D\r^n
g(s,u^n)\psi dxdW\\
&\quad+\int_0^t\int_{|z|_{Z}<1}\int_D\r^n
F\left( u^n(x,s-),z\right)\psi dx\tilde{\pi}(ds,dz)\\
\notag&\quad+\int_0^t\int_{|z|_{Z}\ge1}\int_D\r^n G\left(
u^n(x,s-),z\right)\psi dx\pi(ds,dz),
\end{align}
for all $t\in[0,T]$ and any function $\psi\in X_n$, together with
\begin{equation}\label{3.2}
(\rho^n)_t+(u^n\cdot\nabla)\rho^n=0 \ \ \mbox{in}\ \
Q_T=D\times[0,T],
\end{equation}
\begin{equation}\label{3.3}
u^n|_{t=0}=u_0^n,\; \rho^n|_{t=0}=\rho_0^n \ \ \mbox{in}\ \ D.
\end{equation}
Here we assume that
\begin{equation}\label{3.4}
u_0^n\in X^n,\; u_0^n\to u_0 \ \ \mbox{in} \ \ L^2(D),
\end{equation}
\begin{equation}\label{3.5}
0<m\le\r_0\le M,\; \rho_0^n\to \rho_0 \ \ \mbox{in} \ \ L^\infty (D)
\ \ \mbox{weakly star}.
\end{equation}
Our first step is to solve the transport equation \eqref{3.2} for
$\r^n$ with given $u^n$, it follows from the Proposition
\MyRoman{2}.1 in \cite{DP} that there exists a solution $\r^n$ of
\eqref{3.2} and \eqref{3.3} in $L^\infty(0,T; L^\infty(D))$.
Similarly, from
\begin{equation*}
\left(1/\rho^n\right)_t+u^n\cdot\nabla\left(1/\rho^n\right)=0,
\end{equation*}
we know that $1/\rho^n\in L^\infty(0,T; L^\infty(D))$. Then $\r^n$
has lower and upper bound, that is
\begin{equation}\label{3.7}
0<c\le\r^n\le C.
\end{equation}
Here $c$ and $C$ are independent of $n$ (only on $\r_0$). Next, we
show the existence of a solution $u^n\in \mathbb{D}([0,T];X_n)$ to
\eqref{3.1}. Here $\mathbb{D}([0,T];X_n)$ is the space of all
c\`{a}dl\`{a}g functions $f: [0,T]\rightarrow X_n$. We equip
$\mathbb{D}([0,T];X_n)$ with the Skorokhod topology (see
\cite{MM-1988}). To this end, fix a positive integer $n$, substitute
$\r^n$ into \eqref{3.1} and then look for a function $u^n$ in the
following form
\begin{equation}\label{3.10}
u^n=\sum_{k=1}^n \varphi_k^n(t)w_k(x).
\end{equation}
Choosing $\psi=w_1,w_2,\ldots,w_n$ in \eqref{3.1}, the coefficients
$\varphi_k^n$ satisfies the following stochastic ordinary
differential equations:
\begin{equation}\label{3.11}
\begin{split}
\sum_{k=1}^n&\left(\int_D \rho^n w_k
w_\ell\right)d\varphi_k^n(t)+\sum_{j,k=1}^n\int_D
\rho^n(w_j\varphi_j^n\cdot\nabla)w_k\varphi_k^nw_\ell dxdt\\
&-\int_D \rho^n f(t,\sum_{j=1}^nw_j\varphi_j^n)w_\ell dxdt+\nu\int_D
\sum_{j=1}^n\varphi_j^n(t)\nabla w_j\nabla w_\ell dxdt\\
&=\int_D \rho^n g(t,\sum^n_{j=1} w_j\varphi_j^n(t))w_\ell
dxdW+\int_D\int_{|z|_{Z}<1}\rho^nF\left(
\sum_{j=1}^nw_j\varphi_j^n(t),z\right)w_\ell
\tilde{\pi}(dt,dz)dx\\
&\quad+\int_D\int_{|z|_{Z}\ge1}\rho^nG\left(
\sum_{j=1}^nw_j\varphi_j^n(t),z\right)w_\ell \pi(dt,dz)dx.
\end{split}
\end{equation}
Since the matrix $\left(\int_D \rho^n w_k w_\ell dx\right)$ is
non-degenerate, by virtue of \eqref{3.7},  \eqref{3.11} can be
reformulated as the ordinary form:
\begin{equation}\label{3.12}
\begin{split}
d\varphi_\ell^n+\tilde{F}_\ell(t,\varphi_1^n,\ldots,\varphi_n^n)dt&=\tilde{G}_\ell^n(t,\varphi_1^n,\ldots,\varphi_n^n)dW
+\int_{|z|_{Z}\ge1}\tilde{\Psi}_\ell(\varphi_1^n,\ldots,\varphi_n^n,z)
\pi(dt,dz)\\
&\quad+\int_{|z|_{Z}<1}\tilde{\Phi}_\ell(\varphi_1^n,\ldots,\varphi_n^n,z)\tilde{\pi}(dt,dz),
\end{split}
\end{equation}
with the initial data $\varphi^n_\ell(0)=\varphi^n_{\ell,0}$, where
$\varphi^n_{\ell,0}$ are the coefficients of $u_0^n=\sum_{k=1}^n
\varphi_{k,0}^n w_k$. In view of the assumptions (A) and (B),
$\tilde{F},\tilde{G},\int_{|z|_{Z}\ge1}\tilde{\Psi}\pi(dt,dz)$ and
$\int_{|z|_{Z}<1}\tilde{\Phi}\tilde{\pi}(dt,dz)$ satisfy the Lipschitz and growth conditions.
According to the existence theory \cite[pp.367, Theorem 6.2.3,
Chaper 6]{AD} for the stochastic ordinary differential equations
with jumps, there exists a function
$\varphi_n=(\varphi_1^n,\ldots,\varphi_n^n)$ satisfying equation
\eqref{3.12} and the initial data for a.e. $t\in[0,T_n]$. And then
$u^n$ defined in \eqref{3.10} solves \eqref{3.11} for a.e.
$t\in[0,T_n]$. Next we want to show that we can find a uniform time
interval of existence for all $n$. This will follow from the a
priori estimates established in the next subsection.


\subsection{A priori estimates} \hspace{2mm}

Now, we want to get the needed a priori estimates. To this end,
taking $\psi=w_k$ in \eqref{3.1}, multiplying the result by
$\varphi_k^n(t)$ and then summing over $k=1,2,\ldots,n$, we have
\begin{align}\label{4.1}
\notag\int_D (u^n\rho^ndu^n(t))dx +&\int_D \rho^n u^n
(u^n\cdot\nabla) u^n
dxdt+\nu \int_D \nabla u^n\cdot\nabla u^n dxdt\\
&=\int_D \rho^nf(t,u^n)u^ndxdt+\int_D
\rho^ng(t,u^n)u^ndxdW\\
&\notag\quad+\int_D\int_{|z|_{Z}<1}\rho^n F\left(
u^n(x,t-),z\right)u^n\tilde{\pi}(dt,dz)dx\\
&\notag\quad+\int_D\int_{|z|_{Z}>1}\rho^n G\left(
u^n(x,t-),z\right)u^n \pi(dt,dz)dx.
\end{align}

First, we introduce the following stopping time:
\begin{equation}\label{4.2}
\tau_N= \left\{\begin{array}{lll}
\inf\{t>0:\|\sqrt{\rho^n}u^n(t)\|_{L^2(D)}\ge N\}, &\ \ \mbox{if}\ \
\{\omega\in\Omega:\|\sqrt{\rho^n}u^n(t)\|_{L^2(D)}\ge N
\}\neq\emptyset,\\
T, &\ \ \mbox{if}\ \
\{\omega\in\Omega:\|\sqrt{\rho^n}u^n(t)\|_{L^2(D)}\ge N
\}=\emptyset.
\end{array}\right.
\end{equation}

Define the function $\Phi=\int_D\r^nu^n\cdot u^ndx$. Let
$q=\r^nu^n$, then $\Phi=\int_D {|q|^2\over\r^n} dx$. Note that
$$
\nabla_q \Phi(\r^n,q)=\int_D {2q\over\r^n}dx=\int_D
{2\r^nu^n\over\r^n}dx=\int_D2u^ndx,\; \nabla^2_q
\Phi(\r^n,q)=\int_D{2\over\r^n} \mathbb I \ dx,
$$
and
$$\partial_{\r^n} \Phi(\r^n,q)=\int_D|u^n|^2dx,$$
where $\mathbb I$ is the identity matrix.
 Applying It\^{o}'s formula in Lemma \ref{Ito} to the above function
$\Phi$, from the first equation in \eqref{1.1},
 one deduces that
\begin{align}\label{4.3}
\notag d\int_D|\sqrt{\rho^n}u^n|^2dx&=\int_D u^n u^n {\partial
\rho^n(s)\over \partial s}dxds+2\nu\int_D
u^n\Delta u^ndxds-2\int_D \rho^n u^n(u^n\cdot\nabla)u^ndxds\\
\notag&\quad+2\int_D \rho^n
u^n\left[f(s,u^n)ds+g(s,u^n)dW\right]dx+\int_D
|\sqrt{\rho^n}g(s,u^n)|^2dxds\\
\notag&\quad+\int_D\int_{|z|_Z<1}\left[2\rho^nu^n F\left(
u^n(x,s-),z\right)+\rho^n F^2\left(
u^n(x,s-),z\right)\right]\tilde{\pi}(ds,dz)dx\\
&\quad+\int_D\int_{|z|_Z\ge1}\left[2\rho^nu^n G\left(
u^n(x,s-),z\right)+\r^nG^2\left(
u^n(x,s-),z\right)\right]\tilde{\pi}(ds,dz)dx\\
\notag&\quad+\int_D\int_{|z|_Z<1}|\sqrt{\rho^n}F\left(
u^n(x,s-),z\right)|^2\mu(dz)dx\\
\notag&\quad+\int_D\int_{|z|_Z\ge1}|\sqrt{\rho^n}G\left(
u^n(x,s-),z\right)|^2\mu(dz)dx\\
\notag&\quad+2\int_D\int_{|z|_Z\ge1}\rho^nu^nG\left(
u^n(x,s-),z\right)\mu(dz)dx,
\end{align}
where $s\in[0,t\wedge\tau_N]$, $t\in[0,T_n]$,
$t\wedge\tau_N:=\min\{t,\tau_N\}$. From the second and third
equations in \eqref{1.1}, we can infer that
\begin{equation}\label{4.4}
\begin{split}
0=\int_D
\Dv(u^nu^n\rho^nu^n)dx&=\int_D[u^nu^n\Dv(\rho^nu^n)+\rho^nu^n\nabla(u^nu^n)]dx\\
&=\int_D
[u^nu^n(u^n\cdot\nabla)\rho^n+2\rho^nu^n(u^n\cdot\nabla)u^n]dx,
\end{split}
\end{equation}
where the first equality is due to the condition that $u^n$ vanishes on $(0,T)\times\partial D$. It follows from the
second equation in \eqref{1.1} that
\begin{equation}\label{4.5}
\begin{split}
\int_D u^nu^n{\partial \rho^n(s)\over \partial s} dx=-\int_D u^nu^n
(u^n\cdot\nabla) \rho^ndx=2\int_D \rho^nu^n(u^n\cdot\nabla)u^ndx.
\end{split}
\end{equation}
Substituting \eqref{4.5} into \eqref{4.3}, for all
$s\in[0,t\wedge\tau_N]$, it holds that
\begin{align}\label{4.6}
\notag&\norm{\sqrt{\rho^n}u^n(s)}^2_{L^2(D)}+2\nu\int_0^s
\norm{\nabla
u^n(r)}^2_{L^2(D)}dr\\
\notag&\;\le \norm{\sqrt{\rho_0^n}u_0^n}^2_{L^2(D)}+\int_0^s
2|\langle u^n,\rho^n
f(r,u^n)\rangle|dr+\int_0^s\norm{\sqrt{\rho^n}g(r,u^n)}^2_{L^2(D)}dr\\
&\notag\quad+2\left|\int_0^s \langle u^n,\rho^ng(r,u^n)\rangle
dW\right|+2\left|\int_0^s\int_{|z|_Z\ge1}\langle
u^n,\rho^nG\left(u^n(x,r-),z\right)\rangle\mu(dz)dr\right|\\
&\quad+\int_0^s\int_{|z|_Z<1}\left(2\langle u^n,\rho^n F\left(
u^n(x,r-),z\right)\rangle+\norm{\sqrt{\r^n}F}^2_{L^2(D)}\right)\tilde{\pi}(dr,dz)\\
&\notag\quad +\int_0^s\int_{|z|_Z\ge1}\left(2\langle u^n,
\rho^nG\left(
u^n(x,r-),z\right)\rangle+\norm{\sqrt{\r^n}G}^2_{L^2(D)}\right)\tilde{\pi}(dr,dz)\\
&\notag\quad+\int_0^s\int_D\int_{|z|_Z<1}|\sqrt{\rho^n}F\left(
u^n(x,r-),z\right)|^2\mu(dz)dxdr\\
&\notag\quad+\int_0^s\int_D\int_{|z|_Z\ge1}|\sqrt{\rho^n}G\left(
u^n(x,r-),z\right)|^2\mu(dz)dxdr.
\end{align}
Here $\langle \cdot , \cdot\rangle$ denotes the inner product.
Taking supremum on both sides of \eqref{4.6} over the interval
$[0,t\wedge\tau_N]$, and then taking the mathematical expectation,
we obtain
\begin{align}\label{4.7}
\notag \mathbb E&\sup_{0\le s\le
t\wedge\tau_N}\norm{\sqrt{\rho^n}u^n(s)}^2_{L^2(D)}+2\nu\mathbb
E\int_0^{t\wedge\tau_N} \norm{\nabla
u^n(r)}^2_{L^2(D)}dr\\
\notag&\le
\mathbb E\norm{\sqrt{\rho_0^n}u_0^n}^2_{L^2(D)}+ \mathbb E\int_0^{t\wedge\tau_N}
2|\langle u^n,\rho^n
f(r,u^n)\rangle|dr+ \mathbb E\int_0^{t\wedge\tau_N}\norm{\sqrt{\rho^n}g(r,u^n)}^2_{L^2(D)}dr\\
\notag&\quad+2 \mathbb E\sup_{0\le s\le t\wedge\tau_N}\left|\int_0^s \langle
u^n,\rho^ng(r,u^n)\rangle
dW\right|+2 \mathbb E\left|\int_0^{t\wedge\tau_N}\int_{|z|_Z\ge1}\langle
u^n,G\left(u^n(x,r-),z\right)\rangle\mu(dz)dr\right|\\
&\quad+ \mathbb E\sup_{0\le s\le
t\wedge\tau_N}\int_0^s\int_{|z|_Z<1}\left(2\langle u^n,\rho^n
F\left(
u^n(x,r-),z\right)\rangle+\norm{\sqrt{\r^n}F}^2_{L^2(D)}\right)\tilde{\pi}(dr,dz)\\
\notag&\quad+ \mathbb E\sup_{0\le s\le
t\wedge\tau_N}\int_0^s\int_{|z|_Z\ge1}\left(2\langle u^n,
\rho^nG\left(
u^n(x,r-),z\right)\rangle+\norm{\sqrt{\r^n}G}^2_{L^2(D)}\right)\tilde{\pi}(dr,dz)\\
\notag&\quad+ \mathbb E\int_0^{t\wedge\tau_N}\int_D\int_{|z|_Z<1}|\sqrt{\rho^n}F\left(
u^n(x,r-),z\right)|^2\mu(dz)dxdr\\
\notag&\quad+ \mathbb E\int_0^{t\wedge\tau_N}\int_D\int_{|z|_Z\ge1}|\sqrt{\rho^n}G\left(
u^n(x,r-),z\right)|^2\mu(dz)dxdr\\
\notag&:=I_0+I_1+I_2+I_3+I_4+I_5+I_6+I_7+I_8.
\end{align}

Now, we shall estimate each term in the right-hand side of
\eqref{4.7}. First, for the term $I_1$, by Young's inequality and
the hypothesis on $f$, it holds that
\begin{equation}\label{4.8}
\begin{split}
 I_1&\le
\varepsilon \mathbb E\int_0^{t\wedge\tau_N}\norm{\sqrt{\rho^n}u^n(s)}^2_{L^2(D)}ds+C_\varepsilon \mathbb E\int_0^{t\wedge\tau_N}\norm{\sqrt{\rho^n(s)}f(s,u^n)}^2_{L^2(D)}ds\\
&\le
C\mathbb E\int_0^{t\wedge\tau_N}\norm{\sqrt{\rho^n}u^n(s)}^2_{L^2(D)}ds+C.
\end{split}
\end{equation}
For the term $I_2$, by the assumption on $g$, the H\"{o}lder's
inequality yields
\begin{equation}\label{4.9}
I_2\le C\mathbb E\int_0^{t\wedge\tau_N}
\norm{\sqrt{\rho^n(s)}u^n(s)}^2_{L^2(D)}ds+C.
\end{equation}

Next, we shall estimate the term $I_3$, the hypothesis on $g$ and
Burkholder-Davis-Gundy inequality imply
\begin{equation}\label{4.10}
\begin{split}
I_3&\le C
\mathbb E\left[\int_0^{t\wedge\tau_N}\langle\rho^ng(s,u^n),u^n\rangle^2ds\right]^{1\over2}\\
&\le C
\mathbb E\left[\int_0^{t\wedge\tau_N}\norm{\rho^n(s)}_{L^\infty(D)}\left(1+\norm{u^n(s)}^2_{L^2(D)}\right)\norm{\sqrt{\rho^n}u^n(s)}^2_{L^2(D)}ds\right]^{1\over2}\\
&\le C \mathbb E\sup_{0\le s\le
t\wedge\tau_N}\norm{\sqrt{\rho^n}u^n(s)}_{L^2(D)}\left(\int_0^{t\wedge\tau_N}\norm{\rho^n(s)}_{L^\infty(D)}\left(1+\norm{u^n(s)}^2_{L^2(D)}\right)ds\right)^{1\over2}\\
&\le \varepsilon \mathbb E\sup_{0\le s\le
t\wedge\tau_N}\norm{\sqrt{\rho^n}u^n(s)}^2_{L^2(D)}+C\mathbb E\int_0^{t\wedge\tau_N}\left(1+\norm{\sqrt{\rho^n}u^n(s)}^2_{L^2(D)}\right)ds.
\end{split}
\end{equation}
For the term $I_4$, it follows from the assumption on $G$ and
H\"{o}lder's inequality that
\begin{align}\label{4.11}
\notag I_4&\le \mathbb E\int_0^{t\wedge\tau_N}\left|\int_{|z|_Z\ge1}\langle
\sqrt{\rho^n}u^n,\sqrt{\rho^n}G\left(u^n(x,r-),z\right)\rangle\mu(dz)dr\right|\\
\notag&\le
\mathbb E\int_0^{t\wedge\tau_N}\norm{\sqrt{\rho^n}u^n}^2_{L^2(D)}ds+ \mathbb E\norm{\rho^n}_{L^\infty(D)}\int_0^{t\wedge\tau_N}\left|\int_{|z|_Z\ge1}
\norm{G\left(u^n(x,s-),z\right)}_{L^2(D)}\mu(dz)\right|^2ds\\
\notag&\le
\mathbb E\int_0^{t\wedge\tau_N}\norm{\sqrt{\rho^n}u^n}^2_{L^2(D)}ds+C_{\mu}
\mathbb E\|\rho^n\|_{L^\infty(D)}\int_0^{t\wedge\tau_N}
\int_{|z|_Z\ge1}\norm{G\left(u^n(x,s-),z\right)}^2_{L^2(D)}\mu(dz)ds\\
\notag&\le
\mathbb E\int_0^{t\wedge\tau_N}\norm{\sqrt{\rho^n}u^n(s)}^2_{L^2(D)}ds+C_{\mu} \mathbb E\norm{\rho^n}_{L^\infty(D)}\int_0^{t\wedge\tau_N}\left(1+\norm{u^n(s)}^2_{L^2(D)}\right)ds\\
&\le
(C+1) \mathbb E\int_0^{t\wedge\tau_N}\norm{\sqrt{\rho^n}u^n(s)}^2_{L^2(D)}ds+C.
\end{align}
Here $C_{\mu}=\int_{|z|_Z\ge1}\mu(dz)<\infty$. For the term $I_5$,
in view of the hypothesis on $F$, using the Burkholder-Davis-Gundy
inequality, H\"{o}lder's and Young's inequality, we have
\begin{align}\label{4.12}
\notag I_{5,1}&:=2\mathbb E\sup_{0\le s\le
t\wedge\tau_N}\int_0^s\int_{|z|_Z<1}\langle u^n,\rho^n F\left(
u^n(x,r-),z\right)\rangle\tilde{\pi}(dr,dz)\\
\notag &\le 2\mathbb
E\left[\int_0^{t\wedge\tau_N}\int_{|z|<1}\langle
u^n,\rho^nF(u^n(x,s-),z)\rangle^2\mu(dz)ds
\right]^{1\over2}\\
\notag&\le C\mathbb E\left(\sup_{0\le s\le
t\wedge\tau_N}\norm{\sqrt{\rho^n}u^n(s)}^2_{L^2(D)}\int_0^{t\wedge\tau_N}\int_{|z|_Z<1}\norm{\rho^n}_{L^\infty(D)}
\norm{F(u^n(x,s-),z)}^2_{L^2(D)}\mu(dz)ds\right)^{1\over2}\\
&\le C\mathbb E\left[\sup_{0\le s\le
t\wedge\tau_N}\norm{\sqrt{\rho^n}u^n(s)}_{L^2(D)}\left(\int_0^{t\wedge\tau_N}\left(1+\norm{\sqrt{\rho^n}u^n(s)}^2_{L^2(D)}\right)ds\right)^{1\over2}\right]\\
\notag&\le \varepsilon \mathbb E\sup_{0\le s\le
t\wedge\tau_N}\norm{\sqrt{\rho^n}u^n(s)}^2_{L^2(D)}+C \mathbb E\int_0^{t\wedge\tau_N}\left(1+\norm{\sqrt{\rho^n}u^n(s)}^2_{L^2(D)}\right)ds,
\end{align}
on the other hand, we can infer that
\begin{align}\label{4.12a}
\notag I_{5,2}&:=\mathbb E\sup_{0\le s\le
t\wedge\tau_N}\int_0^s\int_{|z|_Z<1}\norm{\sqrt{\r^n}F}^2_{L^2(D)}\tilde{\pi}(dr,dz)\\
\notag&\le C
\mathbb E\left[\int_0^{t\wedge\tau_N}\int_{|z|_Z<1}\norm{\sqrt{\r^n}F}^4_{L^2(D)}\mu(dz)ds
\right]^{1\over2}\\
&\le C
\mathbb E\left[\int_0^{t\wedge\tau_N}\norm{\r^n}^2_{L^\infty(D)}\int_{|z|_Z<1}\norm{F}^4_{L^2(D)}\mu(dz)ds
\right]^{1\over2}\\
&\notag\le \varepsilon \mathbb E\sup_{0\le s\le
t\wedge\tau_N}\norm{\sqrt{\rho^n}u^n(s)}^2_{L^2(D)}+C\mathbb E\int_0^{t\wedge\tau_N}\left(1+\norm{\sqrt{\rho^n}u^n(s)}^2_{L^2(D)}\right)ds,
\end{align}
then
\begin{equation}\label{4.12b}
I_5\le \varepsilon \mathbb E\sup_{0\le s\le
t\wedge\tau_N}\norm{\sqrt{\rho^n}u^n(s)}^2_{L^2(D)}+C\mathbb E\int_0^{t\wedge\tau_N}\left(1+\norm{\sqrt{\rho^n}u^n(s)}^2_{L^2(D)}\right)ds.
\end{equation}

For the term $I_6$, similarly to $I_5$, by the assumption on $G$,
one has
\begin{equation}\label{4.13}
I_6\le\varepsilon \mathbb E\sup_{0\le s\le
t\wedge\tau_N}\norm{\sqrt{\rho^n}u^n(s)}^2_{L^2(D)}+C\mathbb E\int_0^{t\wedge\tau_N}\left(1+\norm{\sqrt{\rho^n}u^n(s)}^2_{L^2(D)}\right)ds.
\end{equation}

Finally, we shall estimate the last two terms. For the term $I_7$,
in virtue of the assumption (B), using H\"{o}lder's inequality, we
can infer that
\begin{equation}\label{4.14}
\begin{split}
I_7&\le C\mathbb E \int_0^{t\wedge\tau_N}\norm{\rho^n(s)}_{L^\infty(D)}
\int_{|z|_Z<1}\norm{F(u^n(x,s-),z)}^2_{L^2(D)}\mu(dz)ds\\
&\le C\mathbb E\int_0^{t\wedge\tau_N}
\norm{\sqrt{\rho^n}u^n(s)}^2_{L^2(D)}ds+C.
\end{split}
\end{equation}
Similarly to $I_7$, in view of the hypothesis on $G$, one deduces
that
\begin{equation}\label{4.15}
I_8\le  C\mathbb E\int_0^{t\wedge\tau_N}
\norm{\sqrt{\rho^n}u^n(s)}^2_{L^2(D)}ds+C.
\end{equation}
Substituting  \eqref{4.8}-\eqref{4.15} into \eqref{4.7}, for
sufficiently small $\varepsilon>0$, it holds that
\begin{equation}\label{4.16}
\begin{split}
\mathbb E&\sup_{0\le s\le
t\wedge\tau_N}\norm{\sqrt{\rho^n}u^n(s)}^2_{L^2(D)}+2\nu
\mathbb E\int_0^{t\wedge\tau_N}\norm{\nabla u^n(s)}^2_{L^2(D)}ds\\
&\le
\mathbb E\norm{\sqrt{\rho_0^n}u_0^n}^2_{L^2(D)}+C\mathbb E\int_0^{t\wedge\tau_N}\left(1+\norm{\sqrt{\rho^n}u^n(s)}^2_{L^2(D)}\right)ds.
\end{split}
\end{equation}
By the Gronwall inequality, we have
\begin{equation}\label{4.17}
\mathbb E\sup_{0\le s\le
t\wedge\tau_N}\norm{\sqrt{\rho^n}u^n(s)}^2_{L^2(D)}+2\nu
\mathbb E\int_0^{t\wedge\tau_N}\norm{\nabla u^n(s)}^2_{L^2(D)}ds\le C.
\end{equation}
Note that by the property of stopping time, we have
$t\wedge\tau_N\to t$ as $N\to\infty$. Then letting $N\to\infty$ in
\eqref{4.17}, for any $t\in[0,T_n]$, we obtain
\begin{equation}\label{4.18}
\mathbb E\sup_{0\le s\le t}\norm{\sqrt{\rho^n}u^n(s)}^2_{L^2(D)}+2\nu
\mathbb E\int_0^t\norm{\nabla u^n(s)}^2_{L^2(D)}ds\le C.
\end{equation}
Since the constant $C$ is independent of $n$, then $T_n=T$.

Applying It\^{o}'s formula to \eqref{3.1} when $p\ge2$, integrating
over $[0,s], s\in[0,t\wedge\tau_N]$, one has
\begin{equation}\label{4.19}
\begin{split}
&\norm{\sqrt{\rho^n}u^n(s)}^p_{L^2(D)}+p\nu\int_0^s\norm{\sqrt{\rho^n}u^n(r)}^{p-2}_{L^2(D)}\norm{\nabla
u^n(r)}^2_{L^2(D)}dr\\
&\quad=p\int_0^s\norm{\sqrt{\rho^n}u^n(r)}^{p-2}_{L^2(D)}\langle u^n,\rho^ng(r,u^n)\rangle dW\\
&\quad\quad+{p\over
2}\int_0^s\norm{\sqrt{\rho^n}u^n(r)}^{p-2}_{L^2(D)}\left[2\langle
u^n,\rho^nf(r,u^n)\rangle+\norm{\sqrt{\rho^n}g(r,u^n)}^2_{L^2(D)}\right]dr\\
&\quad\quad+{p\over2}\left({p\over2}-1\right)\int_0^s\norm{\sqrt{\rho^n}u^n(r)}^{p-4}_{L^2(D)}\langle
u^n,\rho^ng(r,u^n)\rangle^2dr+\sum_{i=1}^3J_{i,s},
\end{split}
\end{equation}
where
\begin{equation}\label{4.20}
J_{1,s}:=\int_0^s
\int_{|z|_Z\ge1}\left\{\norm{\sqrt{\rho^n}u^n(r)+\sqrt{\rho^n}G(u^n(r-),z)}_{L^2(D)}^p-\norm{\sqrt{\rho^n}u^n(r)}^p_{L^2(D)}\right\}
\pi(dz,dr),
\end{equation}
\begin{equation}\label{4.21}
J_{2,s}:=\int_0^s
\int_{|z|_Z<1}\left\{\norm{\sqrt{\rho^n}u^n(r)+\sqrt{\rho^n}F(u^n(r-),z)}_{L^2(D)}^p-\norm{\sqrt{\rho^n}u^n(r)}^p_{L^2(D)}\right\}
\tilde{\pi}(dz,dr),
\end{equation}
and
\begin{equation}\label{4.22}
\begin{split}
J_{3,s}:=\int_0^s
&\int_{|z|_Z<1}\bigg\{\norm{\sqrt{\rho^n}u^n(r)+\sqrt{\rho^n}F(u^n(r-),z)}_{L^2(D)}^p-\norm{\sqrt{\rho^n}u^n(r)}^p_{L^2(D)}\\
&-p\norm{\sqrt{\rho^n}u^n(r)}_{L^2(D)}^{p-2}\langle\sqrt{\rho^n}u^n,F(u^n(r-),z)\rangle\bigg\}\mu(dz)dr.
\end{split}
\end{equation}
Now, taking supremum up to time $t\wedge\tau_N$ and taking
mathematical expectation in both sides of \eqref{4.19}, and then we
shall estimate each term of the resulting equation. For the first
term, by the Burkholder-Davis-Gundy inequality and H\"{o}lder's
inequality, in virtue of assumption (A) and \eqref{3.7}, we have
\begin{equation}\label{4.22a}
\begin{split}
\mathbb E&\sup_{0\le s\le
t\wedge\tau_N}\left|\int_0^s\norm{\sqrt{\rho^n}u^n(r)}^{p-2}_{L^2(D)}\langle
u^n,\rho^ng(r,u^n)\rangle dW \right|\\
&\leq  C \mathbb E \left[\sup_{0\le s\le
t\wedge\tau_N}\norm{\sqrt{\rho^n}u^n(s)}^{p-2}_{L^2(D)}\left(\int_0^{t\wedge\tau_N}
\langle
u^n,\rho^ng(s,u^n)\rangle^2dt\right)^{1\over2}\right]\\
&\leq C\mathbb E\left[\sup_{0\le s\le
t\wedge\tau_N}\norm{\sqrt{\rho^n}u^n(s)}^{p-1}_{L^2(D)}
\left(\int_0^{t\wedge\tau_N}\norm{\sqrt{\rho^n}g(s,u^n)}^2_{L^2(D)}dt\right)^{1\over2}\right]\\
&\le \varepsilon \mathbb E\sup_{0\le s\le
t\wedge\tau_N}\norm{\sqrt{\rho^n}u^n(s)}^p_{L^2(D)}+C
\mathbb E\left(\int_0^{t\wedge\tau_N}
\left(1+\norm{\sqrt{\rho^n}u^n(s)}^{p}_{L^2(D)}\right)ds\right).
\end{split}
\end{equation}
For the second, third and fourth terms, similar to the first term,
the H\"{o}lder inequality, \eqref{3.7} and assumption (A) yield
\begin{equation}
\begin{split}
\mathbb E&\int_0^{t\wedge\tau_N}\norm{\sqrt{\rho^n}u^n(s)}^{p-2}_{L^2(D)}\langle
u^n,\rho^nf(s,u^n)\rangle ds\\
&\leq C
\mathbb E\int_0^{t\wedge\tau_N}\norm{\sqrt{\rho^n}u^n(s)}^{p-2}_{L^2(D)}\left(1+\norm{\sqrt{\rho^n}u^n(s)}^2_{L^2(D)}\right)ds\\
&\leq C
\mathbb E\int_0^{t\wedge\tau_N}\left(1+\norm{\sqrt{\rho^n}u^n(s)}^p_{L^2(D)}\right)ds,
\end{split}
\end{equation}
and
\begin{equation}
\mathbb E\int_0^{t\wedge\tau_N}\norm{\sqrt{\rho^n}u^n(s)}^{p-2}_{L^2(D)}\norm{\sqrt{\rho^n}g(s,u^n)}^2_{L^2(D)}ds
\leq C
\mathbb E\int_0^{t\wedge\tau_N}\left(1+\norm{\sqrt{\rho^n}u^n(s)}^p_{L^2(D)}\right)ds,
\end{equation}
and
\begin{equation}
\begin{split}
\mathbb E&\int_0^{t\wedge\tau_N}\norm{\sqrt{\rho^n}u^n(s)}^{p-4}_{L^2(D)}\langle
u^n,\rho^ng(s,u^n)\rangle^2 ds\\
&\leq C
\mathbb E\int_0^{t\wedge\tau_N}\norm{\sqrt{\rho^n}u^n(s)}^{p-2}_{L^2(D)}\left(1+\norm{\sqrt{\rho^n}u^n(s)}^2_{L^2(D)}\right)ds\\
&\leq C
\mathbb E\int_0^{t\wedge\tau_N}\left(1+\norm{\sqrt{\rho^n}u^n(s)}^p_{L^2(D)}\right)ds.
\end{split}
\end{equation}
For the term $J_{1,s}$, it follows from the inequality $(a+b)^p\le
2^{p-1}(a^p+b^p)$ for all $p\ge 1$ and $a,b\ge 0$, \eqref{3.7} and
assumption (B) that
\begin{equation}\label{4.23}
\begin{split}
\mathbb E\sup_{0\le s\le t\wedge\tau_N}|J_{1,s}|&\le \mathbb E\sup_{0\le s\le
t\wedge{\tau_N}}\int_0^s\int_{|z|_Z\ge1}\left\{\norm{\sqrt{\rho^n}u^n(r)+\sqrt{\rho^n}G(u^n(r-),z)}_{L^2(D)}^p\right\}
\pi(dz,dr)\\
&\le
2^{p-1}\mathbb E\int_0^{t\wedge{\tau_N}}\int_{|z|_Z\ge1}\left\{\norm{\sqrt{\rho^n}u^n(s)}^p_{L^2(D)}+\norm{\sqrt{\rho^n}G(u^n(s-),z)}_{L^2(D)}^p\right\}
\mu(dz)ds\\
&\le
C(p)\left(\mathbb E(t\wedge{\tau_N})+\int_0^{t\wedge{\tau_N}}\mathbb E\norm{\sqrt{\rho^n}u^n(s)}^p_{L^2(D)}ds\right).
\end{split}
\end{equation}
On the other hand, the terms $J_{2,s}$ and $J_{3,s}$ can be rewrite
as
\begin{align*}
J_{2,s}+J_{3,s}&=\int_0^s
\int_{|z|_Z<1}\bigg\{\norm{\sqrt{\rho^n}u^n(r)+\sqrt{\rho^n}F(u^n(r-),z)}_{L^2(D)}^p-\norm{\sqrt{\rho^n}u^n(r)}^p_{L^2(D)}\\
&\quad-p\norm{\sqrt{\rho^n}u^n(r)}_{L^2(D)}^{p-2}\langle\sqrt{\rho^n}u^n,F(u^n(r-),z)\rangle\bigg\}
\pi (dz,dr)\\
&\quad+\int_0^s
\int_{|z|_Z<1}p\norm{\sqrt{\rho^n}u^n(r)}_{L^2(D)}^{p-2}\langle\sqrt{\rho^n}u^n,F(u^n(r-),z)\rangle\tilde{\pi}
(dz,dr)\\
&:=J_{4,s}+J_{5,s}.
\end{align*}
For all $a, b\in H$ and $p\ge2$, from Taylor's formula, it holds
that
\begin{equation*}
\left||a+b|_H^p-|a|_H^p-p|a|_H^{p-2}\langle a, b\rangle\right|\le
C(p)\left(|a|_H^{p-2}|b|_H^2+|b|_H^p\right).
\end{equation*}
From this above inequality, one has
\begin{equation*}
\begin{split}
\mathbb E\sup_{0\le s\le t\wedge\tau_N}|J_{4,s}|&\le C(P)\mathbb E\sup_{0\le s\le
t\wedge\tau_N}\int_0^s\int_{|z|_Z<1}\bigg\{\norm{\sqrt{\rho^n}u^n(r)}^{(p-2)}_{L^2(D)}\norm{\sqrt{\rho^n}F(u^n(r-),z)}^2_{L^2(D)}\\
&\quad+\norm{\sqrt{\rho^n}F(u^n(r-),z)}^p_{L^2(D)}\bigg\} \pi(dz,dr)\\
&\le C(p)\mathbb E\left(\sup_{0\le s\le
t\wedge\tau_N}\norm{\sqrt{\rho^n}u^n(s)}^{(p-2)}_{L^2(D)}\int_0^{t\wedge\tau_N}\int_{|z|_Z<1}\norm{\sqrt{\rho^n}F(u^n(s-),z)}^2_{L^2(D)}\mu(dz)ds\right)\\
&\quad+C(p)\mathbb E\int_0^{t\wedge\tau_N}\int_{|z|_Z<1}\norm{\sqrt{\rho^n}F(u^n(s-),z)}^p_{L^2(D)}\mu(dz)ds.
\end{split}
\end{equation*}
It follows from the assumption (B) and H\"{o}lder's inequality that
\begin{equation}\label{4.24}
\begin{split}
\mathbb E\sup_{0\le s\le t\wedge\tau_N}|J_{4,s}|&\le {1\over 8} \mathbb E\sup_{0\le
s\le
t\wedge\tau_N}\norm{\sqrt{\rho^n}u^n(s)}^p_{L^2(D)}+C(p)\mathbb E\left[\int_0^{t\wedge\tau_N}\left(1+\norm{\sqrt{\rho^n}u^n(s)}^2_{L^2(D)}\right)ds\right]^{p\over
2}\\
&\quad+C(p)\mathbb E\int_0^{t\wedge\tau_N}\left(1+\norm{\sqrt{\rho^n}u^n(s)}^p_{L^2(D)}\right)ds\\
&\le {1\over 8} \mathbb E\sup_{0\le s\le
t\wedge\tau_N}\norm{\sqrt{\rho^n}u^n(s)}^p_{L^2(D)}+C(p,T)\mathbb E\int_0^{t\wedge\tau_N}\left(1+\norm{\sqrt{\rho^n}u^n(s)}^p_{L^2(D)}\right)ds.
\end{split}
\end{equation}
For the term $J_{5,s}$, by the Burkholder-Davis-Gundy inequality and
Young's inequality, we have
\begin{equation}\label{4.25}
\begin{split}
\mathbb E\sup_{0\le s\le t\wedge\tau_N}|J_{5,s}|&\le
C(p)\mathbb E\Bigg\{\left(\sup_{0\le s\le
t\wedge\tau_N}\norm{\sqrt{\rho^n}u^n(s)}^{2(p-1)}_{L^2(D)}\right)^{1\over2}\\
&\quad\times\left(\int_0^{t\wedge\tau_N}
\int_{|z|_Z<1}\norm{\sqrt{\rho^n}F(u^n(s-),z)}^2_{L^2(D)}\mu(dz)ds\right)^{1\over2}\Bigg\}\\
&\le {1\over8} \mathbb E\sup_{0\le s\le
t\wedge\tau_N}\norm{\sqrt{\rho^n}u^n(s)}^p_{L^2(D)}+C(p)
\mathbb E\left[\int_0^{t\wedge\tau_N}\left(1+\norm{\sqrt{\rho^n}u^n(s)}^2_{L^2(D)}\right)ds\right]^{p\over2}\\
&\le{1\over8} \mathbb E\sup_{0\le s\le
t\wedge\tau_N}\norm{\sqrt{\rho^n}u^n(s)}^p_{L^2(D)}+C(p,T)\mathbb E\int_0^{t\wedge\tau_N}\left(1+\norm{\sqrt{\rho^n}u^n(s)}^p_{L^2(D)}\right)ds.
\end{split}
\end{equation}
Plugging \eqref{4.22a}-\eqref{4.25} into \eqref{4.19}, one deduces
that
\begin{equation}\label{4.26}
\begin{split}
{3\over8} \mathbb E\sup_{0\le s\le
t\wedge\tau_N}&\norm{\sqrt{\rho^n}u^n(s)}^p_{L^2(D)}+C\mathbb E\int_0^{t\wedge\tau_N}\norm{\sqrt{\rho^n}u^n(s)}^{p-2}_{L^2(D)}\norm{\nabla
u^n(s)}^2_{L^2(D)}ds\\
&\le
\mathbb E\norm{\sqrt{\rho_0^n}u_0^n}^p_{L^2(D)}+C(p,T)\left(1+ \mathbb E\int_0^{t\wedge\tau_N}\norm{\sqrt{\rho^n}u^n(s)}^p_{L^2(D)}ds
\right).
\end{split}
\end{equation}
Applying the Gronwall inequality to \eqref{4.26}, we can infer that
\begin{equation}\label{4.27}
\mathbb E\sup_{0\le s\le
t\wedge\tau_N}\norm{\sqrt{\rho^n}u^n(s)}^p_{L^2(D)}\leq C
\left(1+ \mathbb E\norm{\sqrt{\rho_0^n}u_0^n}^p_{L^2(D)}\right).
\end{equation}
Therefore,
\begin{equation}\label{4.28}
\begin{split}
\mathbb E\sup_{0\le s\le
t\wedge\tau_N}\norm{\sqrt{\rho^n}u^n(s)}^p_{L^2(D)}&+C\mathbb E\int_0^{t\wedge\tau_N}\norm{\sqrt{\rho^n}u^n(s)}^{p-2}_{L^2(D)}\norm{\nabla
u^n(s)}^2_{L^2(D)}ds\\
&\leq C\left(1+ \mathbb E\norm{\sqrt{\rho_0^n}u_0^n}^p_{L^2(D)}\right).
\end{split}
\end{equation}
By the argument similar to the case of $p=2$, since
$t\wedge\tau_N\to t$ as $N\to\infty$, letting $N\to\infty$ in
\eqref{4.28}, then
\begin{equation}\label{4.29}
\begin{split}
\mathbb E\sup_{0\le s\le
t}\norm{\sqrt{\rho^n}u^n(s)}^p_{L^2(D)}&+C\mathbb E\int_0^t\norm{\sqrt{\rho^n}u^n(s)}^{p-2}_{L^2(D)}\norm{\nabla
u^n(s)}^2_{L^2(D)}ds\\
&\leq C\left(1+ \mathbb E\norm{\sqrt{\rho_0^n}u_0^n}^p_{L^2(D)}\right).
\end{split}
\end{equation}
Taking the power $p\ge 1$ to \eqref{4.6}, it follows from
\eqref{4.29} that
\begin{equation}\label{4.30}
\mathbb E\left(\int_0^T\norm{\nabla u^n(s)}^2_{L^2(D)}ds\right)^p\le C.
\end{equation}

Next, in order to get the tightness of $\r^nu^n$, we need to
estimate some increments in time of $\rho^nu^n$ in the space
$V^\prime$. For this, we need the following estimates. First, it
follows from \eqref{4.29} and $\rho^n\in L^\infty(0,T; L^\infty(D))$
that
\begin{equation}\label{4.31}
\rho^nu^n\in L^{p}(\Omega; L^\infty(0,T; L^2(D))).
\end{equation}
Then we have
\begin{equation}\label{4.32}
\nabla(\r^nu^n)\in L^{p}(\Omega; L^\infty(0,T; H^{-1}(D))).
\end{equation}
From this, the continuity equation implies that
\begin{equation}\label{4.33}
\partial_t \r^n\in L^{p}(\Omega; L^\infty(0,T; H^{-1}(D))).
\end{equation}
For $p\ge 1$, from \eqref{4.30}, it holds that
\begin{equation*}
u^n\in L^p(\Omega; L^2(0,T; V)).
\end{equation*}
Since $V\hookrightarrow L^6$, then \be\label{4.34} u^n\in
L^p(\Omega; L^2(0,T; L^6(D))), \en and \be\label{4.35} \r^nu^n\in
L^p(\Omega; L^2(0,T; L^6(D))). \en To get the another estimate of
$\r^nu^n$, we need the following lemma in \cite[Theorem 1.1.1]{BL}.
\begin{Lemma}\label{Lemma4.1}
Let $\mathscr{T}$ be a linear operator from $L^{p_1}(0,T)$ into
$L^{p_2}(D)$ and from $L^{q_1}(0,T)$ into $L^{q_2}(D)$ with $q_1\ge
p_1$ and $q_2\le p_2$. Then for any $s\in (0,1), \mathscr{T}$ maps
$L^{r_1}(0,T)$ into $L^{r_2}(D)$, where $r_1={1\over{s/
p_1}+(1-s)/q_1}, r_2={1\over{s/ p_2}+(1-s)/q_2} $.
\end{Lemma}
When $p_1=2, p_2=6, q_1=\infty, q_2=2$ and $s={3\over4}$, it follows
from Lemma \ref{Lemma4.1}, \eqref{4.31} and \eqref{4.35} that
\be\label{4.36} \r^nu^n,\; u^n\in L^p(\Omega; L^{8/3}(0,T; L^4(D))).
\en  By H\"{o}lder's inequality and \eqref{4.36}, it holds that
\be\label{4.37}\begin{split} \int_0^T\left(\int_D (\r^n u^n
u^n)^2dx\right)^{2/3}dt&\le \int_0^T\left(\int_D (\r^n
u^n)^4dx\right)^{1/3} \left(\int_D (u^n)^4 dx\right)^{1/3}dt\\
&\le \left(\left(\int_D (\r^n
u^n)^4dx\right)^{2/3}dt\right)^{1/2}\left(\int_0^T \left(\int_D
(u^n)^4 dx\right)^{2/3}dt\right)^{1/2}\\
&\le C \int_0^T \left(\int_D (\r^nu^n)^4dx\right)^{2/3}dt+C\int_0^T
\left(\int_D (u^n)^4 dx\right)^{2/3}dt\\
&\le C.
\end{split}\en
Then, we have \be\label{4.38} \r^n u^n u^n\in L^p(\Omega;
L^{4/3}(0,T; L^2(D))),\en and \be\label{4.39} \nabla (\r^n u^n
u^n)\in L^p(\Omega; L^{4/3}(0,T; H^{-1}(D))).\en In virtue of the
definition of the norm of $V^\prime$, one deduces that
\begin{equation*}
\norm{\r^nu^n(t+\theta)-\r^nu^n(t)}_{V^\prime}=\sup_{\varphi\in V,\;
\|\varphi\|_V=1}\int_D \left[\r^n
u^n(t+\theta)-\r^nu^n(t)\right]\varphi dx.
\end{equation*}
Thus it follows from \eqref{3.1} that
\begin{equation}\label{4.40}
\begin{split}
\mathbb E\int_0^{T-\theta}\norm{\r^nu^n(t+\theta)-\r^nu^n(t)}^2_{V^\prime}dt&= \mathbb E\int_0^{T-\theta}\left\|\int_t^{t+\theta}
d(\r^nu^n)ds\right\|^2_{V^\prime}dt\\
&\le \mathbb E\int_0^{T-\theta}(J_1+J_2+J_3+J_4+J_5+J_6)dt,
\end{split}
\end{equation}
where
\begin{equation*}
J_1(t):=\left\|\int_t^{t+\theta}\Dv(\r^n u^n\otimes
u^n)ds\right\|^2_{V^\prime},\;
J_2(t):=\left\|\int_t^{t+\theta}\bar{\mu} \Delta u^n
ds\right\|^2_{V^\prime},
\end{equation*}
\begin{equation*}
J_3(t):=\left\|\int_t^{t+\theta}\r^n f(s,u^n)
ds\right\|^2_{V^\prime},\; J_4(t):=\left\|\int_t^{t+\theta}\r^n
g(s,u^n) dW\right\|^2_{V^\prime},
\end{equation*}
\begin{equation*}
\begin{split}
J_5(t):=\left\|\int_t^{t+\theta}\int_{|z|_{Z}<1}\rho^n F\left(
u^n(x,s-),z\right)\tilde{\pi}(ds,dz)\right\|^2_{V^\prime},\\
J_6(t):=\left\|\int_t^{t+\theta}\int_{|z|_{Z}\ge1}\rho^n G\left(
u^n(x,s-),z\right)\pi(ds,dz)\right\|^2_{V^\prime}.
\end{split}
\end{equation*}
For the term $J_1(t)$, one has
\begin{equation*}
\begin{split}
J_1^{1/2}&=\sup_{\varphi\in V;\; \|\varphi\|_{V}=1}\left\{\int_D
\left(\int_t^{t+\theta}\Dv(\r^nu^n\otimes u^n)ds\right)\varphi(x)dx
\right\}\\
&\le C\int_t^{t+\theta}\norm{\r^nu^nu^n}_{L^2(D)}ds.
\end{split}
\end{equation*}
By H\"{o}lder's inequality and \eqref{4.38}, then
\begin{equation}\label{4.41}
\mathbb E\int_0^{T-\theta}J_1(t)dt\lesssim
\theta^{1/2}\left[\mathbb E\left(\int_0^T\norm{\r^nu^nu^n}^{4/3}_{L^2(D)}dt\right)^2\right]^{3/4}\leq
C\theta^{1/2}.
\end{equation}
For the term $J_2(t)$, similarly to \eqref{4.41}, we have
\begin{equation}\label{4.42}
\begin{split}
\mathbb E\int_0^{T-\theta}J_2(t)dt & \le
\mathbb E\int_0^{T-\theta}\left(\int_t^{t+\theta}\norm{\nabla
u^n}_{L^2(D)}ds\right)^2dt \\
& \le \theta
\mathbb E\int_0^{T-\theta}\int_t^{t+\theta}\norm{\nabla
u^n}^2_{L^2(D)}dsdt\leq C\theta.
\end{split}
\end{equation}
For the term $J_3(t)$, by H\"{o}lder's inequality, it follows from
the assumption (A) and \eqref{3.7} that
\begin{equation}\label{4.43}
\begin{split}
\mathbb E\int_0^{T-\theta}J_3(t)dt&\leq C
\mathbb E\int_0^{T-\theta}\left[\int_t^{t+\theta}\norm{\r^n(s)}_{L^\infty(D)}\left(1+\norm{u^n}^2_{L^2(D)}\right)ds\right]^2dt\\
&\leq C\theta \mathbb E\left[\norm{\r^n}^2_{L^\infty(0,T;
L^\infty(D))}\int_0^{T-\theta}\int_t^{t+\theta}\left(1+\norm{u^n(s)}^2_{L^2(D)}\right)dsdt\right] \leq C\theta.
\end{split}
\end{equation}
For the term $J_4$, in virtue of the Burkholder-Davis-Gundy
inequality, H\"{o}lder's inequality and the condition on $g$, we
obtain
\begin{align}\label{4.44}
\notag J_4&\le \int_0^T \mathbb E\left(\sup_{\varphi\in
V,\|\varphi\|_V=1}\int_t^{t+\theta}\int_D\rho^ng(s,u^n)\varphi dxdW \right)^2dt\\
\notag&\le\int_0^T \mathbb E\left(\sup_{\varphi\in
V,\|\varphi\|_V=1}\int_t^{t+\theta}\left(\int_D\rho^ng(s,u^n)\varphi dx\right)^2 ds \right)dt\\
&\le
\int_0^T\left(\mathbb E\int_t^{t+\theta}\norm{\rho^n}^2_{L^\infty(D)}\norm{g(s,u^n)}^2_{L^2(D)}ds\right)dt\\
\notag&\le\int_0^T\left(\mathbb E\int_t^{t+\theta}\norm{\rho^n}^2_{L^\infty(D)}\left(1+\norm{u^n}^2_{L^2(D)}\right)ds\right)dt\\
\notag&\leq C\theta \norm{\rho^n}^2_{L^\infty(D)} \mathbb E\int_0^T\left(1+\norm{u^n}^2_{L^2(D)}\right)dt \leq C\theta.
\end{align}
For the term $J_5$, by the Burkholder-Davis-Gundy inequality and
H\"{o}lder's inequality, the assumption (B) and \eqref{3.7} imply
\begin{equation}\label{4.45}
\begin{split}
J_5&\le \int_0^T \mathbb E\left(\sup_{\varphi\in
V,\|\varphi\|_V=1}\int_t^{t+\theta}\int_{|z|_{Z}<1}\int_D\rho^n
G\left( u^n(x,s-),z\right)\varphi dx\tilde{\pi}(ds,dz) \right)^2dt\\
&\leq C\int_0^T \mathbb E\sup_{\varphi\in
V,\|\varphi\|_V=1}\int_t^{t+\theta}\int_{|z|_{Z}<1}\left(\int_D\rho^n
G\left( u^n(x,s-),z\right)\varphi dx\right)^2\mu(dz)dsdt\\
&\leq C\int_0^T
\mathbb E\int_t^{t+\theta}\int_{|z|_{Z}<1}\norm{\r^n}^2_{L^\infty(D)}\norm{G\left(
u^n(x,s-),z\right)}^2_{L^2(D)}\mu(dz)dsdt\\
&\leq C\int_0^T
\mathbb E\int_t^{t+\theta}\norm{\r^n}^2_{L^\infty(D)}\left(1+\norm{u^n}^2_{L^2(D)}\right)dsdt\\
&\leq C \mathbb E\int_0^T \theta\sup_{0\le t\le
T}\norm{\r^n}^2_{L^\infty(D)}\left(1+\norm{u^n}^2_{L^2(D)}\right)dt\leq
C\theta.
\end{split}
\end{equation}
Finally, for the term $J_6$, similarly to \eqref{4.45}, one has
\begin{equation}\label{4.46}
J_6\leq C \mathbb E\int_0^{T-\t}\sup_{0\le t\le
T}\norm{\r^n}^2_{L^\infty(D)}\left(1+\norm{u^n}^2_{L^2(D)}\right)dt\leq
C\theta.
\end{equation}
Plugging \eqref{4.41}-\eqref{4.46} into \eqref{4.40}, we obtain
\begin{equation}\label{4.47a}
\mathbb E\int_0^{T-\t}\norm{\r^n(t+\t)u^n(t+\t)-\r^n(t)u^n(t)}^2_{V^\prime}dt\le
C\t^{1\over2}.
\end{equation}

Now, we want to get the estimates of the increment of $u^n$ for the
tightness of $u^n$. Note that
\begin{equation}\label{4.49}
\rho^n(t+\theta)\left[u^n(t+\theta)-u^n(t)\right]=\rho^n(t+\theta)u^n(t+\theta)-\rho^n(t)u^n(t)-
u^n(t)\left[\rho^n(t+\theta)-\rho^n(t)\right],
\end{equation}
then we only need to estimate the term
$u^n(t)\left[\rho^n(t+\theta)-\rho^n(t)\right]$. To this end, it
follows from Sobolev embedding theorem, Lemma \ref{Lemma2.1},
\eqref{4.30} and \eqref{4.33} that
\begin{equation}\label{4.50}
\begin{split}
\mathbb E\int_0^{T-\t}&\norm{u^n(t)\left[\rho^n(t+\theta)-\rho^n(t)\right]}_{W^{-1,{3\over2}}(D)}^2dt\\
&\le \mathbb E\int_0^{T-\t}\left(\norm{u^n}_{H^1(D)}\int_t^{t+\theta}\norm{\rho_s^n}_{W^{-1,{3\over2}}(D)}ds\right)^2dt\\
&\leq C\theta^2 \mathbb E\left(\norm{\rho_s^n}^2_{L^\infty(0,T;
H^{-1}(D))}\int_0^T \norm{u^n(t)}^2_{H^1(D)}dt\right) \leq C\theta^2.
\end{split}
\end{equation}
This combining with \eqref{4.47a} and \eqref{4.49} yields
\begin{equation}\label{4.52}
\mathbb E\int_0^{T-\t}\norm{\rho^n(t+\theta)\left[u^n(t+\theta)-u^n(t)\right]}^2_{W^{-1,{3\over2}}(D)}dt\le
C\theta^{1\over2}.
\end{equation}
Since $\rho^n\in L^\infty(0,T; L^\infty(D))$, then
\begin{equation}\label{4.53a}
\mathbb E\int_0^{T-\t}\norm{u^n(t+\theta)-u^n(t)}^2_{W^{-1,{3\over2}}(D)}dt\le
C\theta^{1\over2}.
\end{equation}
\section{Tightness property for the approximation solutions and convergence}

\subsection{Tightness property for the approximation solutions}
\hspace{2mm}

 In this subsection, we shall show the tightness
property for the approximation solutions in the following lemma.
\begin{Lemma}\label{Lemma-Tight} Define
\begin{equation*}
\begin{split}
S&=C(0,T; \R^d)\times M_{\mathbb{N}}(Z\times [0,T])\times
L^{\infty}(0,T; W^{-1,\infty}(D))\\
&\quad\times L^2(0,T; L^2(D))\times L^2(0,T; W^{-\alpha,2}(D)),\;
0<\alpha<\frac{6}{13}
\end{split}
\end{equation*}
equipped with its Borel $\sigma$-algebra.  Let $\Pi^n$ be the
probability on $S$ which is the image of $P$ on $\Omega$ by the map:
$\omega\mapsto
(W_n(\omega,\cdot),\pi_n(\omega,\cdot),\r^n(\omega,\cdot),u^n(\omega,\cdot),\r^nu^n(\omega,\cdot))$,
that is, for any $B\subseteq S$,
\begin{equation*}
\Pi^n(B)=P\left\{\omega\in\Omega:
(W_n(\omega,\cdot),\pi_n(\omega,\cdot),\r^n(\omega,\cdot),u^n(\omega,\cdot),\r^nu^n(\omega,\cdot))\in
B \right\}.
\end{equation*}
Then the family $\Pi^n$ is tight.
\end{Lemma}
\begin{Remark}
We can also prove that $u^n$ is tight in $\mathbb{D}([0,T];H_w)$ as
\cite{MM-1988,ME}, where $\mathbb{D}([0,T];H_w)$ denotes the space
of $H$-valued weakly c\`{a}dl\`{a}g functions.
\end{Remark}
In order to get the tightness of $\r^n, u^n$ and $\r^nu^n$, we
need the following Proposition in \cite{SJ}.
\begin{Proposition}\label{Proposition4.1}
Let $X, B$ and $Y$ be Banach spaces such that $X\subset\subset
B\subset Y$. Assume $1\le p\le\infty$, $\mathcal{E}$ is a set
bounded in $L^p(0,T;X)$ and $\norm{y(t+\t)-y(t)}_{L^p(0,T-\t;Y)}\to
0$ as $\t\to 0$ uniformly for $y\in\mathcal{E}$. Then $\mathcal{E}$
is relatively compact in $L^p(0,T;B)$.
\end{Proposition}
\begin{proof}[Proof of Lemma \ref{Lemma-Tight}] Denote $C(0,\!T; \R^d)\times
L^{\infty}(0,\!T;\! W^{-1,\infty}(D))\times L^2((0,T)\times D)\times
L^2(0,T;\! W^{-\alpha,2}(D))$ by $S_1$.  Let $\Pi_1^n$ be the
probability on $S_1$. We shall find for any $\varepsilon$ subsets:
$\Sigma_\varepsilon\in C(0,T; \R^d), X_\varepsilon\in
L^\infty(0,T;W^{-1,\infty}(D)), Y_\varepsilon\in L^2(0,T;L^2(D))$
and $Z_\varepsilon\in L^2(0,T; W^{-\alpha,2}(D))$ are compact, such
that $P(W_n\notin \Sigma_\varepsilon)\le {\varepsilon\over4},
P(\rho^n\notin X_\varepsilon)\le {\varepsilon\over4}$, $P(u^n\notin
Y_\varepsilon)\le {\varepsilon\over4}$ and $P(\r^nu^n\notin
Z_\varepsilon)\le {\varepsilon\over4}$. We will prove these results
in the following four steps.

{\bf Step 1:} Find a $\Sigma_\varepsilon\in C(0,T; \R^d)$ which is compact,
such that $P(W_n\notin \Sigma_\varepsilon)\le {\varepsilon\over4}$.
To this end. For $\Sigma_\varepsilon$ we rely on classical results
concerning the Brownian motion. For a constant $L_\varepsilon$ to be
chosen later, we consider the set
\begin{equation*}
\Sigma_\varepsilon=\left\{W(\cdot)\in C(0,T; \mathbb{R}^d):
\sup_{t_1,t_2\in[0,T],|t_1-t_2|<{1\over m^6}}m|W(t_2)-W(t_1)|\le
L_\varepsilon,\; \forall m\in\mathbb{N}\right\}.
\end{equation*}
$\Sigma_\varepsilon$ is relatively compact in $C(0,T; \mathbb{R}^d)$
by Arzsela-Ascoli's Theorem. Furthermore $\Sigma_\varepsilon$ is
closed in $C(0,T; \mathbb{R}^d)$. Therefore $\Sigma_\varepsilon$ is
a compact subset of $C(0,T; \mathbb{R}^d)$. We can show that
$P(W_n\notin \Sigma_\varepsilon)\le {\varepsilon\over 4}$. In fact,
by the Chebyshev inequality $P\{\omega: \xi(\omega)\ge r\}\le
{1\over r^k}\mathbb E\left[|\xi(\omega)|^k\right]$, we have
\begin{align*}
P\{\omega: W_n(\omega,\cdot)\notin \Sigma_\varepsilon\}&\le
P\left[\bigcup_{m=1}^\infty\left\{\omega:
\sup_{t_1,t_2\in[0,T],|t_1-t_2|<m^{-6}}|W_n(t_1)-W_n(t_2)|>{L_\varepsilon\over
m}\right\}\right]\\
&\le\sum_{m=1}^\infty\sum^{m^6-1}_{i=0}\left({m\over
L_\varepsilon}\right)^4 \mathbb E\left[\sup_{iTm^{-6}\le t\le
(i+1)Tm^{-6}}\left|W_n(t)-W_n\left(iTm^{-6}\right)\right|^4\right]\\
&\le C\sum_{m=1}^\infty\left({m\over
L_\varepsilon}\right)^4\left(Tm^{-6}\right)^2m^6={C\over
L^4_\varepsilon}\sum_{m=1}^\infty {1\over m^2}.
\end{align*}
Therefore choosing $L^4_\varepsilon={1\over
4C\varepsilon}\left(\sum_{m=1}^\infty {1\over m^2}\right)^{-1}$ we obtain that  $P\{\omega: W_n(\omega,\cdot)\in
\Sigma_\varepsilon\}\ge 1-{\varepsilon\over 4}$.

{\bf Step 2:} To find an $ X_\varepsilon\in
L^\infty(0,T;W^{-1,\infty}(D))$ that is compact, such that
$P(\rho^n\notin X_\varepsilon)\le {\varepsilon\over4}$. For this, we
introduce the space $\Y^1$ with the norm
\begin{equation*}
\norm{y}_{\Y^1}:=\|y\|_{L^\infty(0,T; L^\infty(D))}+\norm{\p
y\over\p t}_{L^\infty(0,T; H^{-1}(D))},
\end{equation*}
then $\Y^1$ is a Banach space. For $q>0$, $\mathbb
E\|y\|^q_{L^\infty(0,T; L^\infty(D))}\le C$ and $\mathbb E\norm{\p
y\over\p t}^q_{L^\infty(0,T; H^{-1}(D))}\le C$, define
$\norm{y}_{\Y^1_{\mathbb E}}$ as the space of random variables $y$
endowed with the norm
\begin{equation*}
\norm{y}_{\Y^1_{\mathbb E}}:=\left(\mathbb E\|y\|^q_{L^\infty(0,T;
L^\infty(D))}\right)^{1\over q}+\left(\mathbb E\norm{\p y\over\p
t}^q_{L^\infty(0,T; H^{-1}(D))}\right)^{1\over q}.
\end{equation*}
We choose $X_\varepsilon$ as a closed ball of radius $r_\varepsilon$
centered at 0 in $L^{\infty}(0,T; W^{-1,\infty}(D))$ with the norm
$\norm{\cdot}_{\Y^1}$. By Proposition \ref{Proposition4.1}, then
$X_\varepsilon$ is compact.

It follows from \eqref{3.7}, \eqref{4.33} and Chebyshev's
inequality that
\begin{equation*}
P(\r^n\notin
X_\varepsilon)=P\left(\norm{\r^n}_{\Y^1}>r_\varepsilon\right)\le
{1\over r_\varepsilon} \mathbb E\left(\norm{\r^n}_{\Y^1}\right)\le
{1\over r_\varepsilon}\norm{y}_{\Y^1_{\mathbb E}} \le {C\over
r_\varepsilon}.
\end{equation*}
Choosing $r_\varepsilon=4C{\varepsilon}^{-1}$, we have $P(\r^n\notin
X_\varepsilon)\le {\varepsilon\over 4}$. Then $P\{\omega:
\r^n(\omega,\cdot)\in X_\varepsilon\}\ge 1-{\varepsilon\over 4}$.

{\bf Step 3:} Find a $Y_\varepsilon\in L^2(0,T;L^2(D))$ that is compact,
such that $P(u^n\notin Y_\varepsilon)\le {\varepsilon\over4}$. For
this, we introduce $\mathcal{Y}^2$ with the norm:
\begin{equation*}
\begin{split}
\|y\|_{\Y^2}&:=\sup_{0\le t\le T}\|y(t)\|_{L^2(D)}+\left(\int_0^T
\|y(t)\|_V^2ds \right)^{1\over2}\\
&\quad+\left(\int_0^{T-\theta}\|y(t+\theta)-y(t)\|^2_{W^{-1,{3\over2}}(D)}dt\right)^{1\over2},
\end{split}
\end{equation*}
then $\Y^2$ is a Banach space. For $1\le p<\infty$, $\mathbb E\sup_{0\le
t\le T}\|y(t)\|^p_{L^2(D)}\le C$, $\mathbb E\left(\int_0^T
\|y(t)\|_V^2ds\right)^{p\over2}\le C$ and $\mathbb E
\int_0^{T-\theta}\|y(t+\theta)-y(t)\|^2_{W^{-1,{3\over2}}(D)}dt\le
C$, define
\begin{align*}
\norm{y}_{\Y^2_{\mathbb E}}&:=\left(\mathbb E\sup_{0\le t\le
T}\|y(t)\|^p_{L^2(D)}\right)^{1\over p}+\left[\mathbb
E\left(\int_0^T
\|y(t)\|_V^2ds \right)^{p\over2}\right]^{2\over p}\\
&\quad+\mathbb E\left(\int_0^{T-\theta}\|y(t+\theta)-y(t)\|^2_{W^{-1,{3\over2}}(D)}dt\right)^{1\over2}.
\end{align*}
Choose $Y_\varepsilon$ as a closed ball of radius
$r^\prime_{\varepsilon}$ centered at 0 in $\mathcal{Y}^2$ with the
norm $\norm{\cdot}_{\Y^2}$. Then Proposition \ref{Proposition4.1}
yields that $Y_\varepsilon$ is compact in $L^2(0,T; L^2(D))$. From
\eqref{4.53a} and Chebyshev's inequality, one deduces that
\begin{equation*}
P(u^n\notin
Y_\varepsilon)=P\left(\norm{u^n}_{\Y^2}>r^\prime_{\varepsilon}\right)\le
{1\over r^\prime_{\varepsilon}} \mathbb
E\left(\norm{u^n}_{\Y^2}\right)\le {1\over
r^\prime_{\varepsilon}}\norm{y}_{\Y^2_{\mathbb E}}\le {C\over
r^\prime_{\varepsilon}}.
\end{equation*}
Choosing $r^\prime_{\varepsilon}=4C{\varepsilon}^{-1}$, we have
$P(u^n\notin Y_\varepsilon)\le {\varepsilon\over 4}$. Then
$P\{\omega: u^n(\omega,\cdot)\in Y_\varepsilon\}\ge
1-{\varepsilon\over 4}$.

{\bf Step 4:} Find a $Z_\varepsilon\in L^2(0,T; W^{-\alpha,2}(D)),
0<\alpha<\frac{6}{13}$ that is compact, such that $P(\r^nu^n\notin
Z_\varepsilon)\le {\varepsilon\over4}$. To this end, define the
space $\Y^3$ with the norm:
\begin{equation*}
\begin{split}
\|y\|_{\Y^3}&:=\|y(t)\|_{L^{8\over3}(0,T;
L^4(D))}+\left(\int_0^{T-\theta}\|y(t+\theta)-y(t)\|^2_{V^\prime}dt\right)^{1\over2},
\end{split}
\end{equation*}
then $\Y^3$ is a Banach space. For $1\le p<\infty$ and $y$ such that
$\mathbb E\|y(t)\|^p_{L^{8\over3}(0,T; L^4(D))}\le C$ and
$\mathbb E\int_0^{T-\theta}\|y(t+\theta)-y(t)\|^2_{V^\prime}dt\le C$, define
\begin{equation*}
\begin{split}
\norm{y}_{\Y^3_{\mathbb E}}&:=\left(\mathbb
E\|y(t)\|^p_{L^{8\over3}(0,T; L^4(D))}\right)^{1\over p}+\mathbb
E\left(\int_0^{T-\theta}\|y(t+\theta)-y(t)\|^2_{V^\prime}dt\right)^{1\over2}.
\end{split}
\end{equation*}
Take $Z_\varepsilon$ as a closed ball of radius
$\tilde r_{\varepsilon}$ centered at 0 in $\Y^3$ with the
norm $\norm{\cdot}_{\Y^3}$. From Proposition \ref{Proposition4.1},
it holds that $Z_\varepsilon$ is compact in $L^2(0,T;
W^{-\alpha,2}(D))$ . On the other hand, \eqref{4.40} and Chebyshev's
inequality imply
\begin{equation*}
P(\r^nu^n\notin Z_\varepsilon)=P\left(\norm{\r^nu^n}_{\Y^3}>\tilde
r_{\varepsilon}\right)\le {1\over \tilde r_{\varepsilon}} \mathbb
E\left(\norm{\r^nu^n}_{\Y^3}\right)\le {1\over \tilde
r_{\varepsilon}}\norm{y}_{\Y^3_{\mathbb E}}\le {C\over \tilde
r_{\varepsilon}}.
\end{equation*}
Choosing $\tilde r=4C{\varepsilon}^{-1}$, we
have $P(\r^nu^n\notin Z_\varepsilon)\le {\varepsilon\over 4}$. Then
$P\{\omega: \r^nu^n(\omega,\cdot)\in Z_\varepsilon\}\ge
1-{\varepsilon\over 4}$.

To summarize, we can find suitable $\t$ such that
\begin{equation*}
P\left\{\omega: W_n\in \Sigma_\varepsilon,\rho^n\in X_\varepsilon,
u^n\in Y_\varepsilon, \r^nu^n\in Z_\varepsilon\right\}\ge
1-\varepsilon.
\end{equation*}
Hence
\begin{equation}\label{4.53}
\Pi_1^n(\Sigma_\varepsilon\times X_\varepsilon\times
Y_\varepsilon\times Z_\varepsilon)\ge 1-\varepsilon.
\end{equation}
Since $M_{\mathbb{N}}(Z\times [0,T])$ endowed with the Prohorov's
metric is a separable metric space, by Theorem 3.2 in \cite[pp.
29]{PKR}, then it holds that the distributions of the family
$\{\pi_n, n\in\mathbb{N}\}$ are tight on $M_{\mathbb{N}}(Z\times
[0,T])$. Therefore, it follows from the Corollary in \cite[pp.16,
Corollary1.3]{KHJ} that the distribution of the joint processes
$$\left\{(W_n(\omega,\cdot),\pi_n(\omega,\cdot),\r^n(\omega,\cdot),u^n(\omega,\cdot),\r^nu^n(\omega,\cdot)): n\in\mathbb{N}\right\}$$
are tight on $S$. The tightness property of $\Pi^n$ is thus proved.
\end{proof}
\subsection{Convergence}
\hspace{2mm}

In this subsection, we shall pass to the limit directly to get the
solution. The following proof differs from the approach of D. Parto
and Zabczyk in \cite{PZ-1992}. It is based on the method used by A.
Bensoussan in \cite{BA}. In the above Subsection, we know that $\Pi^n$
is tight on the space $S=C(0,T; \R^d)\times M_{\mathbb{N}}(Z\times
[0,T])\times L^{\infty}(0,T; W^{-1,\infty}(D))\times L^2(0,T;
L^2(D))\times L^2(0,T; W^{-\alpha,2}(D)),\; 0<\alpha<\frac{6}{13}$.
According to Jakubowski-Skorohod's theorem, there exists a
subsequence such that $\Pi^{n_j}\stackrel{w}{\to}\Pi$, where $\Pi$
is a probability on $S$. Moreover, there exist a probability space
$(\mathring{\Omega},\mathring{\mathscr{F}},\mathring{P})$ and random
variables
$(\mathring{W}_{n_j},\mathring{\pi}_{n_j},\mathring{\rho}^{n_j},
\mathring{u}^{n_j},\mathring{\rho}^{n_j}\mathring{u}^{n_j})$;
$(W,\pi,\rho, u,h)$ with values in $S$, such that the probability
distribution of
$(\mathring{W}_{n_j},\mathring{\pi}_{n_j},\mathring{\rho}^{n_j},
\mathring{u}^{n_j},\mathring{\rho}^{n_j}\mathring{u}^{n_j})$ is
$\Pi^{n_j}$ and the probability distribution of $(W,\pi,\rho, u,h)$
is $\Pi$, and
\begin{equation}\label{4.54}
(\mathring{W}_{n_j},\mathring{\pi}_{n_j},\mathring{\rho}^{n_j},\mathring{u}^{n_j},\mathring{\rho}^{n_j}\mathring{u}^{n_j})\to
(W, \pi, \rho, u, h) \ \ \mbox{in}\ \ S, \mathring{P}-a.s.
\end{equation}
Here $W$ is
$\mathring{\mathscr{F}}_t=\sigma\{\rho(s),u(s),W(s),\pi(s)\}_{0\le
s\le t}$ standard Brownian motion. In fact, we need to prove that
for $s\le t$ and $i^2=-1$
\begin{equation}\label{4.55}
\mathring{\mathbb E}\left[\exp\left\{i\lambda(W(t)-W(s)\right\}\right]=\exp\left(-{\lambda^2\over2}(t-s)\right).
\end{equation}
It is sufficient to show that
\begin{equation}\label{4.56}
\mathring{\mathbb E}\left[\exp\left\{i\lambda(W(t)-W(s)\right\}|\mathring{\mathscr{F}}_t\right]=\exp\left(-{\lambda^2\over2}(t-s)\right).
\end{equation}
Here $\mathring{\mathbb E}$ denotes the mathematical expectation with
respect to the probability space
$(\mathring{\Omega},\mathring{\mathscr{F}},\mathring{P})$. Note that
if $X$ is $\mathring{\mathscr{F}}$ measurable and
$\mathring{\mathbb E}(|Y|), \mathring{\mathbb E}(|XY|)<\infty$, then
\begin{equation}\label{4.57}
\mathring{\mathbb E}(XY|\mathring{\mathscr{F}})=X\mathring{\mathbb E}(Y|\mathring{\mathscr{F}}),\;
\mathring{\mathbb E}(\mathring{\mathbb E}(Y|\mathring{\mathscr{F}}))=\mathring{\mathbb E}(Y).
\end{equation}
Thus
\begin{equation}\label{4.58}
\mathring{\mathbb E}(XY)=\mathring{\mathbb E}\left(X\mathring{\mathbb E}(Y|\mathring{\mathscr{F}})\right).
\end{equation}
Using \eqref{4.58}, we can prove \eqref{4.56} if the
following equality is satisfied.
\begin{equation}\label{4.59}
\mathring{\mathbb E}\left[\exp\{i\lambda(W(t)-W(s))\}\Lambda(\rho,u,W,\pi)\right]=\exp\left(-{\lambda^2(t-s)\over2}\right)\mathring{\mathbb E}\left(\Lambda(\rho,u,W,\pi)\right),
\end{equation}
for any continuous bounded functional $\Lambda(\rho,u,W,\pi)$ on $S$
depending only on the values of $\rho, u, W, \pi$ on $(0,s)$. Since
$\mathring{W}_{n_j}(t)-\mathring{W}_{n_j}(s)$ is independent of
$\Lambda(\mathring{\rho}^{n_j},\mathring{u}^{n_j},\mathring{W}_{n_j},\mathring{\pi}_{n_j})$
and $\mathring{W}_{n_j}$ is a Brownian motion, then we have
\begin{equation}\label{4.60}
\begin{split}
\mathring{\mathbb E}&\left[\exp\{i\lambda(\mathring{W}_{n_j}(t)-\mathring{W}_{n_j}(s))\}\Lambda(\mathring{\r}^{n_j},\mathring{u}^{n_j},\mathring{W}_{n_j},\mathring{\pi}_{n_j})\right]\\
&=\mathring{\mathbb E}\left[\exp\{i\lambda(\mathring{W}_{n_j}(t)-\mathring{W}_{n_j}(s))\}\right]\mathring{\mathbb E}(\Lambda(\mathring{\r}^{n_j},\mathring{u}^{n_j},\mathring{W}_{n_j},\mathring{\pi}_{n_j}))\\
&=\exp\left(-{\lambda^2(t-s)\over2}\right)E(\Lambda(\mathring{\r}^{n_j},\mathring{u}^{n_j},\mathring{W}_{n_j},\mathring{\pi}_{n_j})).
\end{split}
\end{equation}
Taking $j\to\infty$, \eqref{4.54} and the continuity of $\Lambda$
imply \eqref{4.59}. Then $W(t)$ is a $\mathring{\mathscr{F}}_t$
standard Brownian motion.

For a random measure $\eta$ on $Z\times [0,T]$ and for any
$A\in\mathscr{Z}$, where $\mathscr{Z}$ is the Borel sets on $Z$,
define the measure valued process $N_\eta(t)$ by
$N_\eta(t)=\left\{A\mapsto
N_\eta(t,A):=\eta(A\times(0,t])\right\},\; t\in[0,T]$. We have the
following proposition:
\begin{Proposition}
$\mathring{\pi}_{n_j}$ and $\pi$ are time homogeneous Poisson random
measures on $\mathscr{B}(Z)\times\mathscr{B}([0,T])$ over
$(\mathring{\Omega},\mathring{\mathscr{F}},\mathring{P})$ with
intensity measure $\mu$.
\end{Proposition}
\begin{proof} We shall prove the Proposition by the definition. Since
$\mathring{\pi}_n$ and $\pi_n$ have the same distribution and
$\pi_n$ is a time homogeneous Poisson random measure, from
\cite[Proposition A.5, Remark A.6]{BH}, it holds that
$\mathring{\pi}_n$ satisfies (1)-(3) of Definition \ref{defA.2}.

Therefore, we only need to prove that $\mathring{\pi}_n$ satisfies
(4) of Definition \ref{defA.2} with the filtration
$\mathring{\mathscr F}_t=\sigma\left\{W(s),\pi(s),\r^n(s),
u^n(s),\r, u; 0\le s\le t\right\}, t\in[0,T]$. To this end, fix
$n\in\mathbb{N}, t_0\in[0,T]$ and $t_0\le s\le t$. The definition of
$\mathring{\mathscr F}_t$ implies that $\mathring{\pi}_n$ is
$\mathring{\mathscr F}_t$-adapted. Then it remains to prove that
$\mathring{X}_n=N_{\mathring{\pi}_n}(t)-N_{\mathring{\pi}_n}(s)$ is
independent of $\mathring{\mathscr F}_{t_0}$.

It follows from (2) of Definition \ref{defA.2} that the random
variable$\mathring{X}_n=N_{\mathring{\pi}_n}(t)-N_{\mathring{\pi}_n}(s)$
is independent of $N_{\mathring{\pi}_n}(t_0)$, hence we only need to
prove that $\mathring{X}_n$ is independent of $\mathring{\r}^n(r),
\mathring{u}^n(r)$ and $\r(r), u(r)$ for any $r\le t_0$.

Fix $r\in[0,t_0]$. Since the distributions of
$(\mathring{W}_n,\mathring{\pi}_n,\mathring{\rho}^n,
\mathring{u}^n)$ and $(W_n,\pi_n,\r^n,u^n)$ are same, then
$\mathcal{L}(\mathring{\r}^n|_{[0,r]},\mathring{u}^n|_{[0,r]},\mathring{X}_n)=\mathcal{L}(\r^n|_{[0,r]},u^n|_{[0,r]},X_n)$,
where $X_n=N_{\pi_n}(t)-N_{\pi_n}(s)$ and $\mathcal{L}(h)$ denotes
the distribution of $h$. Note that $\mathring{\pi}_n=\pi$ (see
\cite[Theorem D.1]{BH}), $(\r^n, u^n)$ is the solution of the
stochastic approximation equations, then it is adapted to the
$\sigma$-algebra generated by $\mathring{\pi}_n$. Therefore,
$\r^n|_{[0,r]}, u^n|_{[0,r]}$ are independent of $X_n$. From the in
\cite[Remark A.6]{BH} and
$\mathcal{L}(\mathring{\r}^n|_{[0,r]},\mathring{u}^n|_{[0,r]},\mathring{X}_n)=\mathcal{L}(\r^n|_{[0,r]},u^n|_{[0,r]},X_n)$,
one can deduce that
$\mathring{\r}^n|_{[0,r]},\mathring{u}^n|_{[0,r]}$ are independent
of $\mathring{X}_n$.  By \cite[Lemma 9.3]{BH}, it holds that
$\mathring{X}_n$ is independent of $\r|_{[0,r]}, u|_{[0,r]}$. Since
$\pi(\omega)=\mathring{\pi}_n(\omega)$ for all $\omega\in\Omega$ and
$n\in\mathbb{N}$, then $\pi$ is a time homogeneous Poisson random
measure.
\end{proof}

 Now we need to prove that $(\mathring{W}_{n_j},
\mathring{\pi}_{n_j},\mathring{\rho}^{n_j},\mathring{u}^{n_j})$
satisfies the equation \eqref{3.1}, that is,
\begin{align}\label{4.61}
\notag\mathbb{P}^{n}[\mathring{\r}&^{n_j}\mathring{u}^{n_j}(t)]+\int_0^t\mathbb{P}^{n}\left[\Dv(\mathring{\rho}^{n_j}\mathring{u}^{n_j}\otimes
\mathring{u}^{n_j})-\nu\Delta \mathring{u}^{n_j}\right]ds\\
&=\mathbb{P}^{n}[\mathring{\r}^{n_j}\mathring{u}^{n_j}(0)]+\mathbb{P}^{n}\left[\int_0^t
\mathring{\r}^{n_j}f(s,\mathring{u}^{n_j})ds+\int_0^t
\mathring{\r}^{n_j}g(s,\mathring{u}_{n_j})d\mathring{W}_{n_j}\right]\\
\notag&\quad+\mathbb{P}^n\int_0^t\int_{|z|_{Z}<1}\mathring{\r}^{n_j}
F\left(
\mathring{u}^{n_j}(x,s-),z\right)\tilde{\mathring{\pi}}_{n_j}(ds,dz)\\
\notag&\quad+\mathbb{P}^n\int_0^t\int_{|z|_{Z}\ge1}\mathring{\r}^{n_j}
G\left( \mathring{u}^{n_j}(x,s-),z\right)
\mathring{\pi}_{n_j}(ds,dz),
\end{align}
where $\mathbb{P}^n: L^2(D)\to X_n$ is the projection onto $X_n$. To
this end, we define
\begin{align*}
\xi^n(t)&=\mathbb{P}^{n}[\rho^nu^n(t)-\rho^nu^n(0)]+\int_0^t\mathbb{P}^n\left[\Dv(\rho^nu^n\otimes
u^n)-\nu\Delta u^n\right]ds\\
&\quad -\mathbb{P}^n\left[\int_0^t\rho^nf(s,u^n)ds+\int_0^t
\rho^ng(s,u^n)dW_n\right]\\
&\quad-\mathbb{P}^n\int_0^t\int_{|z|_{Z}<1}\r^n F\left(
u^n(x,s-),z\right)\tilde{\pi}_n(ds,dz)\\
&\quad-\mathbb{P}^n\int_0^t\int_{|z|_{Z}\ge1}\r^n G\left(
u^n(x,s-),z\right) \pi_n(ds,dz),
\end{align*}
and
\begin{equation*}
Z^n=\int_0^T \|\xi^n(t)\|^2_{H^{-1}(D)}dt.
\end{equation*}
Of course
\begin{equation*}
Z^n=0, \; P-a.s.
\end{equation*}

 Let
\begin{align*}
\mathring{\xi}^{n_j}(t)&=\mathbb{P}^{n}[\mathring{\rho}^{n_j}\mathring{u}^{n_j}(t)-\mathring{\rho}^{n_j}\mathring{u}^{n_j}(0)]+\int_0^t\mathbb{P}^{n}\left[\Dv(\mathring{\rho}^{n^j}\mathring{u}^{n^j}\otimes
\mathring{u}^{n_j})-\nu\Delta \mathring{u}^{n_j}\right]ds\\
&\quad-\mathbb{P}^{n}\left[\int_0^t\mathring{\rho}^{n_j}f(s,\mathring{u}^{n_j})ds+\int_0^t
\mathring{\rho}^{n_j}g(s,\mathring{u}^{n_j})d\mathring{W}_{n_j}\right]\\
&\quad-\mathbb{P}^n\int_0^t\int_{|z|_{Z}<1}\mathring{\r}^{n_j}
F\left(
\mathring{u}^{n_j}(x,s-),z\right)\tilde{\mathring{\pi}}_{n_j}(ds,dz)\\
&\quad-\mathbb{P}^n\int_0^t\int_{|z|_{Z}\ge1}\mathring{\r}^{n_j}
G\left(
\mathring{u}^{n_j}(x,s-),z\right)\mathring{\pi}_{n_j}(ds,dz).
\end{align*}
and
\begin{equation*}
Y^{n_j}=\int_0^T\|\mathring{\xi}^{n_j}(t)\|^2_{H^{-1}(D)}dt.
\end{equation*}
We have the following proposition:
\begin{Proposition}
$Y^{n_j}=0 \ \ \mathring{P}-\mbox{a.s.}$, that is,
$(\mathring{W}_{n_j},
\mathring{\pi}_{n_j},\mathring{\rho}^{n_j},\mathring{u}^{n_j})$
satisfies the equation \eqref{3.1}.
\end{Proposition}
\begin{proof}The difficulty is that $Z^n$ is not expressed as a
deterministic function of  $(W_n, \pi_n, \r^n, u^n)$  because of the
presence of the stochastic integral. By Theorem 2.4 and Corollary
2.5 in \cite{BH-2011}, we can infer that
\begin{equation}\label{4.62a}
\mathcal{L}(\rho^n, u^n, \xi_n, W_n,
\pi_n)=\mathcal{L}(\mathring{\r}^{n_j}, \mathring{u}^{n_j},
\mathring{\xi_{n_j}}, \mathring{W}_{n_j}, \mathring{\pi}_{n_j}).
\end{equation}
 Note that $Y^{n_j}$ is continuous as a
function of $\mathring{\xi_{n_j}}$ if $\mathring{u}^{n_j}$ belongs
to a finite dimensional subspace of $H^1(D)$. In view of
\eqref{4.62a} and the continuity of $Y^{n_j}$, one deduces that the
distribution of $Y^{n_j}$ is equal to the distribution of $Z^{n_j}$
on $\mathbb{R}_+$, that is,
\begin{equation}\label{4.62b}
\mathring{\mathbb E}\phi(Y^{n_j})= \mathbb E\phi(Z^{n_j}),
\end{equation}
for any $\phi\in C_b(\mathbb{R}_+)$, where $C_b(X)$ is the space of
continuous bounded functions defined on $X$. Now, let
$\varepsilon>0$ be an arbitrary number and $\phi_\varepsilon\in
C_b(\mathbb{R}_+)$ defined by
\begin{equation*}
\phi_\varepsilon=
\begin{cases}
{y\over\varepsilon},\; 0\le y<\varepsilon;\\
1,\; y\ge \varepsilon.
\end{cases}
\end{equation*}
One can check that
\begin{equation*}
\begin{split}
\mathring{P}(Y^{n_j}\ge \varepsilon)&=\int_{\mathring{\Omega}}
1_{[\varepsilon,\infty]}Y^{n_j}d\mathring{P}\le
\int_{\mathring{\Omega}}
1_{[0,\varepsilon]}{Y^{n_j}\over\varepsilon}d\mathring{P}+\int_{\tilde{\Omega}}
1_{[\varepsilon,\infty]}Y^{n_j}d\mathring{P}.
\end{split}
\end{equation*}
Hence by the definition of $\mathring{\mathbb E}(Y^{n_j})$, we can infer
that
\begin{equation*}
\mathring{P}(Y^{n_j}\ge \varepsilon)\le
\mathring{\mathbb E}\phi_\varepsilon(Y^{n_j}),
\end{equation*}
which, together with \eqref{4.62b} implies that
\begin{equation*}
\mathring{P}(Y^{n_j}\ge \varepsilon)\le \mathbb E\phi_\varepsilon(Z^{n_j}).
\end{equation*}
By the fact that $(\rho^n, u^n, W_n, \pi_n)$ satisfies the Galerkin
equation, from the above inequality, it holds that
\begin{equation}\label{4.62c}
\mathring{P}(Y^{n_j}\ge \varepsilon)\le
\mathbb E\phi_\varepsilon(Z^{n_j})=0,
\end{equation}
for any $\varepsilon>0$. Since $\varepsilon>0$ is arbitrary, from
\eqref{4.62c}, we can infer that
\begin{equation}\label{4.62}
Y^{n_j}=0 \ \ \mathring{P}-\mbox{a.s.}.
\end{equation}
It follows from \eqref{4.62} that $(\mathring{W}_{n_j},
\mathring{\pi}_{n_j},\mathring{\rho}^{n_j},\mathring{u}^{n_j})$
satisfies the equation \eqref{3.1}.
\end{proof}

Now, we want to pass to the limit directly. To this end, we need the
following proposition and lemma (see \cite[Chapter 3]{KO}).
\begin{Proposition}[Uniformly integrable]\label{Propo2}
If there exists a nonnegative  measurable function $f$ in
$\mathbb{R}^+$, such that $\lim\limits_{x\to\infty} {f(x)\over
x}=\infty$ and $\sup_{t\in T}\mathbb E[f(|X_t|)]<\infty$. Then $X_t$ is a
set of uniformly integrable.
\end{Proposition}

\begin{Lemma}[Vitali's convergence Theorem]\label{Lemma2}
Suppose $p\in[1,\infty)$, $X_n\in L^p$ and $X_n$ converges to $X$ in
probability. Then the following are equivalent:

(1) $X_n\stackrel{L^p}{\to}X$;

(2) $|X_n|^p$ is uniformly integrable;

(3) $\mathbb E(|X_n|^p)\to \mathbb E(|X|^p)$.
\end{Lemma}

Let us denote the subsequence
$(\mathring{\rho}^{n_j},\mathring{u}^{n_j}),j\ge1$ by
$(\mathring{\rho}^n,\mathring{u}^n)$. Since
$(\mathring{\rho}^n,\mathring{u}^n)$ and $(\rho^n,u^n)$ has the same
distribution, thus by \eqref{4.29} and \eqref{4.30}, we have
\begin{equation}\label{4.68}
\sup_{n}\mathring{\mathbb E}\left(\sup_{0\le s \le
T}\norm{\sqrt{\mathring{\rho}^n}\mathring{u}^n}^{2p}_{L^2(D)}\right)\le
C,
\end{equation}
\begin{equation}\label{4.69}
\sup_{n}\mathring{\mathbb E}\left(\int_0^T\norm{\nabla
\mathring{u}^n(s)}^2_{L^2(D)}ds\right)\le C.
\end{equation}


\begin{proof}[Proof of Theorem \ref{Theorem1.1}]
In order to prove Theorem \ref{Theorem1.1}, we will break the limits
into deterministic and stochastic parts. First, we pass the limits
of the deterministic parts and finally pass the limits of the
stochastic parts.

\noindent{\bf Taking the limits of deterministic parts:} Note that
$\mathring{\r}^n\in L^4(\mathring{\Omega}; L^\infty(0,T;
L^\infty(D)))$, then
\begin{equation}\label{4.76}
\mathring{\r}^n\to \r \ \ \mbox{weakly star in} \ \
L^2(\mathring{\Omega}; L^\infty(0,T; L^\infty(D))),
\end{equation}
and
\begin{equation}\label{4.77}
\mathring{\mathbb E}\sup_{t\in[0,T]}\norm{\mathring{\r}^n}^4_{W^{-1,\infty}(D)}\le
C.
\end{equation}
This together with Proposition \ref{Propo2}, \eqref{4.54} and
Vitali's convergence Theorem implies
\begin{equation}\label{4.78}
\mathring{\r}^n\to \r \ \ \mbox{strongly in} \ \
L^2(\mathring{\Omega}; L^\infty(0,T; W^{-1,\infty}(D))).
\end{equation}
By \eqref{4.69} we can infer that the sequence $\mathring{u}^n$
contains a subsequence, still denoted by $\mathring{u}^n$, that satisfies
\begin{equation}\label{4.70}
\mathring{u}^n\to u \ \ \mbox{weakly in}\ \ L^2(\mathring{\Omega};
L^2(0,T; H^1(D))).
\end{equation}
Similar, by the fact $\mathring{\r}^n(\omega)\in L^\infty(0,T;
L^\infty(D))$, in view of \eqref{4.68}, it holds that
\begin{equation}\label{4.71}
\mathring{u}^n\to u \ \ \mbox{weakly star in}\ \
L^4(\mathring{\Omega}; L^\infty(0,T; L^2(D))).
\end{equation}
Let us consider the positive nondecreasing function $f(x)=x^4$ in
Proposition \ref{Propo2}. The function $f(x)$ obviously satisfies
$\lim\limits_{x\to\infty} {f(x)\over x}=\infty$. Thanks to the
estimate
$\mathring{\mathbb E}\sup_{t\in[0,T]}\norm{\mathring{u}^n}^4_{L^2(D)}\le C$,
we have that $\sup_{n\ge
1}\mathring{\mathbb E}\sup_{t\in[0,T]}f(\norm{\mathring{u}^n}_{L^2(D)})\le
C$. By Proposition \ref{Propo2}, we see that the family
$\left\{\norm{\mathring{u}^n}_{L^2(D)}: n\in\mathbb{N}\right\}$ is
uniformly integrable with respect to the probability measure. From
Vitali's convergence Theorem and \eqref{4.54}, one deduces that
\begin{equation}\label{4.72}
\mathring{u}^n\to u \ \ \mbox{strongly in} \ \
L^2(\mathring{\Omega}; L^2(0,T; L^2(D))).
\end{equation}
Next, from \eqref{4.68}, it holds that
\begin{equation}\label{4.74}
\mathring{\r}^n\mathring{u}^n\to h \ \ \mbox{weakly star in} \ \
L^2(\mathring{\Omega}; L^\infty(0,T; L^2(D))).
\end{equation}
By \eqref{4.68}, we have
$\mathring{\mathbb E}\sup_{t\in[0,T]}\norm{\mathring{\r}^n\mathring{u}^n}^4_{L^2(D)}\le
C$. This yields that
$\mathring{\mathbb E}\sup_{t\in[0,T]}\norm{\mathring{\r}^n\mathring{u}^n}^4_{W^{-\alpha,2}(D)}\le
C$. It follows from \eqref{4.54} and Vitali's convergence Theorem
that
\begin{equation}\label{4.75}
\mathring{\r}^n\mathring{u}^n\to h \ \ \mbox{strongly in} \ \
L^2(\mathring{\Omega}; L^2(0,T; W^{-\alpha,2}(D))).
\end{equation}

From the fact that $\norm{fg}_{W^{-1,6}(D)}\le C
\norm{f}_{W^{-1,\infty}(D)}\norm{g}_{H^1(D)}$ in Lemma
\ref{Lemma2.1},  \eqref{4.78} and \eqref{4.70}, then one has
\begin{equation}\label{4.80}
\mathring{\r}^n\mathring{u}^n\to\r u \ \ \mbox{weakly in} \ \
L^2(\mathring{\Omega}; L^2(0,T; W^{-1,6}(D))).
\end{equation}
This together with \eqref{4.75} implies that
\begin{equation}\label{4.81}
h=\r u.
\end{equation}
It follows form \eqref{4.38} that
\begin{equation}\label{4.79}
\mathring{\r}^n\mathring{u}^n\mathring{u}^n\to \bar{h} \ \
\mbox{weakly in}\ \ L^2(\mathring{\Omega}; L^{4/3}(0,T; L^2(D))).
\end{equation}
Similarly, by the fact that $\norm{fg}_{W^{-1,3/2}(D)}\le
C\norm{f}_{W^{-1,2}(D)}\norm{g}_{H^1(D)}$ in Lemma \ref{Lemma2.1},
and \eqref{4.70}, \eqref{4.75} and \eqref{4.81}, we can
infer that
\begin{equation}\label{4.82}
\mathring{\r}^n\mathring{u}^n\mathring{u}^n\to\r uu \ \ \mbox{weakly
in} \ \ L^2(\mathring{\Omega}; L^1(0,T; W^{-1,3/2}(D))).
\end{equation}
By \eqref{4.72}, the continuity of $f$ and $g$, Vitali's convergence
Theorem imply that
\begin{equation}\label{4.83}
f(t,\mathring{u}^n)\to f(t,u) \ \ \mbox{strongly in} \ \
L^2(\mathring{\Omega}; L^2(0,T; L^2(D))),
\end{equation}
and
\begin{equation}\label{4.84}
g(t,\mathring{u}^n)\to g(t,u) \ \ \mbox{strongly in} \ \
L^2(\mathring{\Omega}; L^2(0,T; L^2(D))).
\end{equation}

\noindent{\bf Taking the limits of stochastic parts:} First, we show
that
\begin{equation}\label{4.85}
\int_0^t
\mathring{\r}^ng(s,\mathring{u}^n(s))d\mathring{W}_n(s)\rightharpoonup
\int_0^t\rho g(s,u(s))dW(s) \ \ \mbox{in} \ \ L^2(\mathring{\Omega};
L^2(0,T; L^2(D))).
\end{equation}
To deal with the stochastic integral, we introduce the function:
\begin{equation}\label{4.86}
\tilde{G}_{\varepsilon}(t)={1\over\varepsilon}\int_0^t
J\left({t-s\over\varepsilon}\right)\tilde{G}(s)ds,
\end{equation}
where $J$ is the standard mollifier.
Note that
\begin{equation}\label{4.87}
\mathring{\mathbb E}\int_0^T\norm{\tilde{G}_{\varepsilon}(t)}_{L^2(D)}^2dt\le
\mathring{\mathbb E}\int_0^T\norm{\tilde{G}(t)}_{L^2(D)}^2dt,
\end{equation}
and
\begin{equation}\label{4.88}
\tilde{G}_\varepsilon(t)\to \tilde{G}(t) \ \ \mbox{in} \ \ L^2(0,T;
L^2(D)).
\end{equation}
Since $\int_0^t \tilde{G}^n(s)d\mathring{W}_n(s)\in
L^2(\mathring{\Omega}; L^2(D))$, then for $\forall \phi\in
L^2(\mathring{\Omega}; L^2(D))$, there exists a $\xi$ such that
\begin{equation}\label{4.94}
\mathring{\mathbb E}\left\langle\phi,\int_0^t
\tilde{G}^n(s)d\mathring{W}_n(s)\right\rangle\to
\mathring{\mathbb E}\left\langle\phi,\xi\right\rangle.
\end{equation}
Next, we will show that $\xi=\int_0^t \tilde{G}(s)dW(s)$.
Integrating by parts, we obtain
\begin{equation}\label{4.89}
\int_0^t
\tilde{G}^n_\varepsilon(s)d\mathring{W}_n(s)=\tilde{G}^n_\varepsilon(t)\mathring{W}_n(t)-\int_0^t
[\tilde{G}^n_\varepsilon(s)]_s\mathring{W}_n(s)ds.
\end{equation}
Letting $n\to\infty$ in \eqref{4.89}, in virtue of \eqref{4.54} and
\eqref{4.84}, we can obtain
\begin{equation}\label{4.90}
\int_0^t \tilde{G}^n_\varepsilon(s)d\mathring{W}_n(s)\rightharpoonup
\tilde{G}_\varepsilon(t)W(t)-\int_0^t[\tilde{G}_\varepsilon(s)]_sW(s)ds=\int_0^t\tilde{G}_\varepsilon(s)dW(s).
\end{equation}
Since
\begin{equation}\label{4.91}
\begin{split}
\mathring{\mathbb E}\norm{\int_0^t
\tilde{G}^n_\varepsilon(s)d\mathring{W}_n(s)}_{L^2(D)}^2&=\mathring{\mathbb E}\left(\int_0^t
\|\tilde{G}^n_\varepsilon(s)\|^2_{L^2(D)}ds\right)\\
&\le
\mathring{\mathbb E}\left(\int_0^t\|\mathring{\rho}^ng(s,\mathring{u}^n(s))\|^2_{L^2(D)}ds\right)\le
C,
\end{split}
\end{equation}
then it follows from Remark \ref{Remark2.1} that
\begin{equation}\label{4.92}
\int_0^t \tilde{G}^n_\varepsilon(s)d\mathring{W}_n(s)\rightharpoonup
\int_0^t\tilde{G}_\varepsilon(s)dW(s) \ \ \mbox{in}\ \
L^2(\mathring{\Omega}; L^2(D)).
\end{equation}
That is, $\forall \phi\in L^2(\mathring{\Omega}; L^2(D))$, we have
\begin{equation}\label{4.93}
\mathring{\mathbb E}\left\langle\phi,\int_0^t
\tilde{G}^n_\varepsilon(s)d\mathring{W}_n(s)\right\rangle\to
\mathring{\mathbb E}\left\langle\phi,\int_0^t\tilde{G}_\varepsilon(s)dW(s)\right\rangle.
\end{equation}
 Note that
\begin{align}\label{4.95}
\notag\mathring{\mathbb E}&\left\langle\phi,\int_0^t
\tilde{G}^n(s)d\mathring{W}_n(s)\right\rangle-E\left\langle\phi,\int_0^t
\tilde{G}(s)dW(s)\right\rangle\\
\notag&=\mathring{\mathbb E}\left\langle\phi,\int_0^t[\tilde{G}^n(s)-\tilde{G}^n_\varepsilon(s)]d\mathring{W}_n(s)\right\rangle\\
&\quad+
\mathring{\mathbb E}\left\langle\phi,\int_0^t\tilde{G}^n_\varepsilon(s)d\mathring{W}_n(s)-\int_0^t
\tilde{G}_\varepsilon(s)dW(s)\right\rangle\\
\notag&\quad+\mathring{\mathbb E}\left\langle\phi,\int_0^t
[\tilde{G}_\varepsilon(s)-\tilde{G}(s)]dW(s)\right\rangle\\
\notag&:=H_1+H_2+H_3.
\end{align}
For the first term $H_1$, the Cauchy-Schwarz inequality,
\eqref{4.84} and \eqref{4.87} yield that
\begin{equation}\label{4.96}
\begin{split}
\mathring{\mathbb E}&\left\langle\phi,\int_0^t[\tilde{G}^n(s)-\tilde{G}^n_\varepsilon(s)]d\mathring{W}_n(s)\right\rangle\\
&\le
\left(\mathring{\mathbb E}\|\phi\|^2_{L^2(D)}\right)^{1\over2}\left(\mathring{\mathbb E}\norm{\int_0^t[\tilde{G}^n(s)-\tilde{G}^n_\varepsilon(s)]d\mathring{W}_n(s)}_{L^2(D)}^2\right)^{1\over2}\\
&\le
\left(\mathring{\mathbb E}\|\phi\|^2_{L^2(D)}\right)^{1\over2}\mathring{\mathbb E}\left(\int_0^t\norm{\tilde{G}^n(s)-\tilde{G}^n_\varepsilon(s)}^2_{L^2(D)}ds\right)^{1\over2}
\to 0,
\end{split}
\end{equation}
as $n\to\infty$ and $\varepsilon\to 0$. Similar as $H_1$, for the
term $H_3$, it follows from the Cauchy-Schwarz inequality and
\eqref{4.88} that $\mathring{\mathbb E}\left\langle\phi,\int_0^t
[\tilde{G}_\varepsilon(s)-\tilde{G}(s)]dW(s)\right\rangle\to 0$ as
$n\to\infty$ and $\varepsilon\to 0$. This together with
\eqref{4.93}, \eqref{4.95} and \eqref{4.96} implies that
$\xi=\int_0^t \tilde{G}(s)dW(s)$. The proof of \eqref{4.85} is thus
complete.

Next, we show that
\begin{equation}\label{4.97}
\int_0^t\int_{|z|_{Z}<1}\mathbb{P}^n\left[\mathring{\r}^n F\left(
\mathring{u}^n(x,s-),z\right)\right]\tilde{\mathring{\pi}}_n(ds,dz)\rightharpoonup\int_0^t\int_{|z|_{Z}<1}\r
F\left(u(x,s-),z\right)\tilde{\pi}(ds,dz)
\end{equation}
in $M^2(\mathring{\Omega},[0,T],L^2(D))$, which is the space of all
$\mathring{\mathscr{F}}_t$-martingales $M_t$ such that
$\mathring{\mathbb E}\int_0^T\norm{M_t}^2_{L^2(D)}dt<\infty$. From
\eqref{4.72} and the continuity of $F\left(
\mathring{u}^n(x,s-),z\right)$, we can infer that
$\mathbb{P}_n\left[F\left( \mathring{u}^n(x,s-),z\right)\right]$
converges to $F\left(u(x,s-),z\right)$ in $L^2(Z,\mu; L^2(D))$
almost everywhere
$(\mathring{\omega},s)\in\mathring{\Omega}\times[0,T]$. Thanks to
the convergence \eqref{4.76}, it holds that
\begin{equation}\label{4.99}
\mathbb{P}_n\left[\mathring{\r}^nF\left(
\mathring{u}^n(x,s-),z\right)\right]\rightharpoonup \r
F\left(u(x,s-),z\right) \ \ \mbox{in}\ \
L^2(\mathring{\Omega}\times[0,T]; L^2(Z,\mu; L^2(D))).
\end{equation}
On the other hand, for any $\psi\in L^2(\mathring{\Omega}\times[0,T];
L^2(Z,\mu; L^2(D)))$, one has
\begin{align}\label{4.100}
\notag\int_0^t&\int_{|z|_{Z}<1} \left\langle
\mathbb{P}_n\left[\mathring{\r}^nF\left(
\mathring{u}^n(x,s-),z\right)\right],\psi\right\rangle\tilde{\mathring{\pi}}_n(ds,dz)-\int_0^t\int_{|z|_{Z}<1}
\left\langle\r
F\left(u(x,s-),z\right),\psi\right\rangle\tilde{\pi}(ds,dz)\\
&=\int_0^t\int_{|z|_{Z}<1}\left\langle
\mathbb{P}_n\left[\mathring{\r}^nF\left(
\mathring{u}^n(x,s-),z\right)\right],\psi\right\rangle(\tilde{\mathring{\pi}}_n-\tilde{\pi})(ds,dz)\\
\notag&\quad+\int_0^t\int_{|z|_{Z}<1}\left\{\left\langle
\mathbb{P}_n\left[\mathring{\r}^nF\left(
\mathring{u}^n(x,s-),z\right)\right]-\r
F\left(u(x,s-),z\right),\psi\right\rangle\right\}\tilde{\pi}(ds,dz).
\end{align}
Note that all the integrals in \eqref{4.100} are well-defined thanks to the
discussion above. Since $\mathring{\pi}_n=\pi$ for any $n$ almost
surely, by \cite[Proposition B.1]{BH}, we can infer
that the first term on the right-hand side of \eqref{4.100} goes to
0 as $n\to\infty$. It follows from the continuity of the stochastic
integral (as linear functional from $M^2([0,T],L^2(Z,\mu; L^2(D)))$
into $M^2(\mathring{\Omega}\times[0,T]; L^2(D))$) and \eqref{4.99}
that the second term on the right-hand side of \eqref{4.100} also
converges to zero as $n\to\infty$. Similarly, thanks to
$\int_{|z|_Z\ge1} |z|^p \mu(dz)<\infty,\forall p\ge 1$, one deduces
that  $\int_0^t\int_{|z|_{Z}\ge1}\mathbb{P}^n\left[\mathring{\r}^n
G\left(\mathring{u}^n(x,s-),z\right)\right]\mathring{\pi}_n(ds,dz)\rightharpoonup\int_0^t\int_{|z|_{Z}\ge1}\r
G\left(u(x,s-),z\right)\pi(ds,dz)$. The proof of Theorem
\ref{Theorem1.1} is thus complete.
\end{proof}

\bigskip

\section*{Acknowledgments}

The work of R. M. Chen was partially supported by the Simons Foundation under Grant 354996 and the NSF grant
DMS-1613375. 
The research of D. Wang was supported in part by the National Science Foundation under grants DMS-1312800 and DMS-1613213.
 H. Wang's research was supported in part by the National Natural Science
Foundation of China-NSAF (No. 11271305).

\end{document}